\documentclass[a4paper,10pt]{article}
\usepackage[utf8]{inputenc}
\usepackage{amsmath,amsfonts,amssymb,amscd,amsthm,xspace, graphicx}
\usepackage[normalem]{ulem}
\usepackage[mathscr]{eucal}
\usepackage{color}
\usepackage{ulem}
\usepackage[table]{xcolor}

\usepackage[pagebackref,colorlinks,linkcolor={blue},citecolor={red}]{hyperref}
\newtheoremstyle{plain}{\topsep}{\topsep}{\slshape}{}{\bfseries}{.}{.5em}{}
\newtheorem{theorem}{Theorem}
\newtheorem{lemma}[theorem]{Lemma}

\newtheorem{proposition}[theorem]{Proposition}

\newtheorem{mainthm}{Theorem}

\theoremstyle{remark}
\newtheorem{remark}[theorem]{Remark}
\numberwithin{equation}{section}
\newtheorem{definition}[theorem]{Definition}

\def\R{{\mathbb R}}

\def\G{\mathcal{G}}
\def\H{\mathcal{H}}
\def\e{\varepsilon}
\def\E{\widetilde{E}}
\def\EE{\mathcal{E}}

\def\({\left(}
\def\){\right)}
\def\supp{\mathrm{supp \,}}
\def\S{\mathcal{S}}
\def\T{\mathcal{T}}
\def\SS{\widetilde\S}

\def\Z{\mathcal{Z}}

\DeclareMathOperator*{\esssup}{ess\,sup}

\begin{document}

%%%%%openings
\title{Quantitative partial regularity of the Navier-Stokes equations and applications}
\author{ Zhen Lei\footnotemark[1]  \and Xiao Ren\footnotemark[2] }
\renewcommand{\thefootnote}{\fnsymbol{footnote}}
\footnotetext[1]{School of Mathematical Sciences, LMNS and Shanghai
Key Laboratory for Contemporary Applied Mathematics, Fudan University, China.  \ \ \ \ Email: zlei@fudan.edu.cn}
\footnotetext[2]{School of Mathematical Sciences,
Fudan University, Shanghai 200433, P. R.China. \ \ \ \ Email:  xiaoren18@fudan.edu.cn}
\maketitle
%%%%%

\begin{abstract} We prove a logarithmic improvement of the Caffarelli-Kohn-Nirenberg partial regularity theorem  for the Navier-Stokes equations. The key idea is  to find a quantitative counterpart  for the absolute continuity of the dissipation energy using the pigeonhole principle. Based on the same method, for any suitable weak solution, we show the existence of intervals of regularity in one spatial direction with length depending exponentially on the natural local energies of the solution. Then, we give two applications of the latter result in the axially symmetric case. The first one is a  local quantitative regularity criterion for suitable weak solutions with small swirl. The second one is a slightly improved one-point CKN criterion which implies all known (slightly supercritical) Type I regularity results in the literature.
\end{abstract}

\section{Introduction}
We study the incompressible Navier-Stokes equations in three dimensions:
\begin{equation}  \tag{NS}\label{eq-NS}
\left\{
\begin{aligned}
  \partial_t v - \Delta v + (v \cdot \nabla) v + \nabla p &= 0,\\
  \nabla \cdot v &= 0,\\
\end{aligned}
\right.
\end{equation} 
where $v,p$ are the unknown velocity and pressure fields. Whether the solutions to  \eqref{eq-NS} can develop finite time singularities from regular initial data remains a question of central importance in the theory of partial differential equations. It has been called one of the seven most important unsolved problems by the Clay Mathematics Institute, see \cite{FeffermanClay}. 

In 1934, based on the natural energy structure of \eqref{eq-NS}, Leray \cite{Leray1934} constructed global weak solutions for the Cauchy problem of \eqref{eq-NS} with finite kinetic energy $\sup_{t\ge 0} \int_{\R^3} |v|^2 dx$ and finite dissipation energy (Dirichlet integral) 
$\int_0^{+\infty} \int_{\R^3} |\nabla v|^2 dx dt.$ Moreover, for these solutions the global energy inequality 
\begin{equation}
\int_{\R^3 \times \{t\}} |v|^2 dx + 2 \int_{0}^t \int_{\R^3} |\nabla v|^2 dxdt \le \int_{\R^3 \times \{0\}} |v|^2 dx
\end{equation}
holds for a.a. $t > 0$. Later, Hopf \cite{Hopf} obtained a similar existence result for the boundary value problem of \eqref{eq-NS} on bounded domains. As well-known, \eqref{eq-NS} enjoys a natural scaling invariance property, namely, for any solution $(v,p)$ to \eqref{eq-NS}, the rescaled functions
\begin{equation} \label{eq-scaling}
v^{(\lambda)}(x,t) = \lambda v(\lambda x, \lambda^2 t), \ p^{(\lambda)}(x,t) = \lambda^2 p(\lambda x, \lambda^2 t)
\end{equation}
also satisfy \eqref{eq-NS} for any $\lambda > 0$. The kinetic and dissipation energies of \eqref{eq-NS} are both supercritical with respect to the natural scaling of \eqref{eq-NS}, namely, the energies of the zoomed-in solution $v^{(\lambda)}$ with $\lambda < 1$ are greater than those of the original solution $v$. This indicates that the a priori estimates of the Leray-Hopf weak solutions are too weak to control the behavior of $v$ on fine scales. 

In the papers \cite{Scheffer1, Scheffer2}, Scheffer  initiated the study of partial regularity for the Navier-Stokes equations, that is, to estimate the dimension of the potential singular set of a weak solution. He considered a subclass of the Leray-Hopf weak solutions satisfying the local energy inequality, and proved that the singular set is of Hausdorff dimension no more than $\frac53$. In the landmark work of Caffarelli, Kohn and Nirenberg \cite{CKN}, Scheffer's result is improved, and it is shown that the 1d parabolic Hausdorff measure of the singular set is equal to $0$ for a class of suitable weak solutions. \cite{CKN} is a foundation for many important results later, \emph{e.g.}, the $L^\infty_t L^3_x$ regularity criterion \cite{ESS}, backward and forward self-similar solutions \cite{RusinSverak, JiaInvent}, regularity via Liouville theorems \cite{SereginActa}, \emph{etc}. Simplified proofs of the partial regularity theorem are given by Lin \cite{Lin}, Ladyzhenskaya-Seregin \cite{LadySer} and Vasseur \cite{Vasseur}. To describe the theory precisely, we recall a few important notions.

\begin{definition}
We say that a point $w = (x,t)$ is a \emph{regular} point of $v$, if $v$ is bounded (and H\"older continuous) in $Q(r,w)$ for some $r > 0$. Here $Q(r,w) = Q(r)+w$ and $Q(r) = B(r) \times (-r^2, 0)$. We say that $w$ is a \emph{singular} point of $v$ if it is not regular. We denote the singular set of $v$ by $\mathcal{S}[v]$. For a solution $v$ in the spacetime cylinder $Q(1)$, it is understood that $\mathcal{S}[v] \subset Q(1) \cup \(B(1) \times \{0\}\)$.
\end{definition}

\begin{definition} \label{def-2}
 We say that a pair of functions $(v,p)$ is a \emph{local suitable weak solution} to the Navier-Stokes equations in the domain $\mathcal{D} \subset \R^3 \times \R$, if the following conditions hold.
\begin{enumerate}
\item $v \in L_t^\infty L_x^2(\mathcal{D})$, $\nabla v \in L_t^2L_x^2(\mathcal{D})$, $p \in L_t^{\frac32}L_x^{\frac32}(\mathcal{D})$.
\item $v,p$ satisfy the Navier-Stokes equations in $\mathcal{D}$ in the sense of distributions.
\item $v,p$ satisfy the local energy inequality
\begin{align} \label{eq-local-energy}
\int_{\mathcal{D}_T} |v|^2 \phi dx \ + & \ 2 \int_{-\infty}^T \int_{\mathcal{D}_t} |\nabla v|^2 \phi dxdt \le  \nonumber\\
&\le\int_{-\infty}^T \int_{\mathcal{D}_t} \left[|v|^2(\phi_t + \Delta \phi) + (|v|^2 + 2 p)v \cdot \nabla \phi \right]dxdt.
\end{align}
for any nonnegative function $\phi \in C_0^\infty(\mathcal{D})$. Here, $\mathcal{D}_t$ is the time $t$ slice of $\mathcal{D}$, \emph{i.e.}, $\mathcal{D}_t = \mathcal{D} \cap (\R^3 \times \{t\})$.
\end{enumerate}
 (Note that this is the definition from \cite{Lin}, and the assumption on pressure is stronger than that in \cite{CKN}.)
\end{definition}

 Although it is not clear whether all Leray-Hopf weak solutions are suitable, in the appendix of \cite{CKN} the authors were able to construct suitable Leray-Hopf weak solutions for the Cauchy problem and the boundary value problem of \eqref{eq-NS} with general $L^2$ initial data. See also \cite[Lemma 2.3]{Lin} or \cite[Theorem 1.1]{LadySer} for the a priori estimates on the pressure. 

\begin{definition}
For any non-negative function $f(r)$ defined on some interval $[0,\delta_0)_r$ with $f(0) = 0$, we define the parabolic $f$-Hausdorff measure on $\R^3 \times \R$ as
\begin{equation}
\mathcal{P}^f (\mathcal{S}) = \lim_{\delta \to 0+}  \mathcal{P}^f_{\delta} (\mathcal{S})
\end{equation}
with
\begin{equation}
\mathcal{P}^f_{\delta} (\mathcal{S}) = \inf \left\{ \sum\limits_{i=1}^{+\infty} f(r_i) : \big\{Q(r_i,w_i)\big\}_{i=1}^{+\infty} \ \mbox{is a covering of} \ \S, \ 0 \le r_i \le \delta \right\}
\end{equation}
where $Q(r_i, w_i) = B(r_i) \times (-r_i^2,0) + w_i$. $\mathcal{P}^f$ is a special case of Caratheodory's construction, see \cite[Section 2.10]{Federer}. For $f(r) = r^\alpha$, this is the parabolic $\alpha$-dimensional Hausdorff measure introduced in \cite{CKN}, and we simply write $\mathcal{P}^f = \mathcal{P}^\alpha$.
\end{definition}

The classical partial regularity theorem of Caffarelli-Kohn-Nirenberg reads

\begin{theorem}[CKN partial regularity] \label{thm-partial-reg-old}
For a local suitable weak solutuion $(v,p)$ to the Navier-Stokes equations, $\mathcal{P}^1(\mathcal{S}[v]) = 0$. 
\end{theorem}

The proof of Theorem \ref{thm-partial-reg-old} is based on the important $\e$-regularity criterion.  Using the a priori bounds in Definition \ref{def-2}  and interpolation, we know that $v \in L_t^{s}L_x^{l}(Q(1))$ for any $s ,l$  satisfying $\frac{2}{s}+\frac{3}{l} = \frac{3}{2}$, $l \in [2, 6]$.  We define the functionals
\begin{equation}
\G = \G[v,p] \stackrel{\Delta}= \int_{Q(1)} \left(|v|^3 + |p|^{3/2} \right) dxdt < +\infty
\end{equation}
and
\begin{equation}
\H = \H[v] \stackrel{\Delta}= \int_{Q(1)} |\nabla v|^2 dxdt < +\infty
\end{equation}
for a suitable weak solution $(v,p)$. By \eqref{eq-local-energy}, the kinetic and the dissipation energies of $v$ on $Q(r)$ for any $r<1$ are bounded in terms of $\G[v,p]$.

\begin{theorem}[local $\e$-regularity] \label{thm-ereg}
Let $v$ be a local suitable weak solution to the Navier-Stokes equations in $Q(1)$, there exists an absolute positive constant $\e_*$ such that if
\begin{equation}
\G[v,p] \le \e_{*},
\end{equation}
then $v$ is H\"older continuous in $\overline{Q\left(\frac12\right)}$, and satisfies the pointwise estimates
\begin{equation}
 \|\nabla_x^m v\|_{L^\infty(Q(1/2))} \lesssim_m 1, \quad m = 0,1,2,\cdots.
\end{equation}
\end{theorem}

We refer to \cite{CKN, Lin, LadySer, Vasseur} for proofs of the H\"older continuity of $v$ in Theorem \ref{thm-ereg}, and to \cite[Proposition 2.1]{SverakActa} for a remark on how to obtain the higher regularity estimates. Note that the methods in \cite{Lin, LadySer} rely on certain (non-quantitative) compactness arguments, while the method in \cite{Vasseur} based on the De Giorgi iteration is quantitative. With the help of Theorem \ref{thm-ereg} and  an iteration of scale-invariant quantities, the one-point regularity criterion Theorem \ref{thm-ckn-criterion} can be established. Then, Theorem \ref{thm-partial-reg-old} is  a corollary of Theorem \ref{thm-ckn-criterion} and the absolute continuity of the dissipation energy $\int |\nabla v|^2 dxdt$.

\begin{theorem}[CKN criterion] \label{thm-ckn-criterion}
Let $v$ be a local suitable weak solution to the Navier-Stokes equations in $Q(1)$. There exists a universal positive constant $\e_{**}$ such that
\begin{equation}
\limsup_{r \to 0+} \frac{1}{r} \int_{Q(r)} |\nabla v|^2 dxdt \le \e_{**} \implies  (0,0) \notin \mathcal{S}[v].
\end{equation}
\end{theorem}

We point out that the proof of Theorem \ref{thm-partial-reg-old} in \cite{CKN} is still non-quantitative, essentially because the absolute continuity of the dissipation energy is a non-quantitative fact. Given any integrable function $g$ on $Q(1)$, it is well-known that for any $\e > 0$, there exists $\delta > 0$ such that $\int_U |g| dxdt < \e$ for any $U$ with $|U| \stackrel{\Delta}= \int_U 1 dxdt < \delta$. However, it is clearly impossible to give an explicit dependence of $\delta$ on $\epsilon$ unless higher integrability of $g$ is assumed. 

Remarkably, Choe and Lewis \cite{ChoeLewis} improved Theorem \ref{thm-partial-reg-old} for the first time and showed that $\mathcal{P}^{f} (\mathcal{S}[v]) = 0$ for $f(r) = r|\ln r|^\alpha$ for a small positive $\alpha$ (precisely $0 \le \alpha<\frac{3}{44}$). The key step is to show that $\ln (1/d) |\nabla v|^2$ is locally integrable where $d$ is the parabolic distance to certain subsets of the singular set. See also the work \cite{RWW} for an improvement on the range of $\alpha$. 

For any singular point $w = (x,t) \in S[v]$, we say that $w$ is Type I if
\begin{equation} \label{eq-type1-condition}
\limsup_{r \to 0+}\frac{1}{r} \int_{Q(r,w)} |\nabla v|^2 dxdt < + \infty.
\end{equation}
For $w = (0,0)$, the pointwise upper bound
\begin{equation} \label{eq-type1-rate}
|v| \le \frac{C_*}{\sqrt{|x|^2-t}}, \ \mbox{for} \ (x,t) \in Q(r), \ r>0
\end{equation}
would imply \eqref{eq-type1-condition}. It is a major open problem in the field to rule out the possibility of blow-up under the condition \eqref{eq-type1-condition} or \eqref{eq-type1-rate}. In other words, it is of great importance to improve the CKN criterion Theorem \ref{thm-ckn-criterion}, in particular, to \emph{remove the smallness requirement} there.

%In this paper, we will take a different approach and show that $\mathcal{P}^f(\S[v]) = 0$ for $f(r) = r |\ln r|$.
%In this paper, we will show that $\alpha$ can be taken as 1.

\subsection{Main results}
\subsubsection{Quantitative partial regularity}

In this paper, we first give a further improvement on the partial regularity theorem using a different approach from \cite{ChoeLewis}. 

\begin{mainthm}[Improved partial regularity] \label{thm-partial-reg}
Let $v$ be a local suitable weak solution to the Navier-Stokes equation in $Q(1)$, then we have, for $f(r) = r |\ln r|$,
\begin{equation} \label{eq-PfSv-0}
\mathcal{P}^{f}(\mathcal{S}[v]) = 0.
\end{equation}

\end{mainthm}

\begin{remark}
This result remains true for the forced Navier-Stokes equations with external force $g$ belonging to $L^q_{x,t}$, $q>\frac52$.  Such an external force term can be treated by introducing $r^{3q-5}\int_{Q(r)} |g|^q dxdt$ in the iteration of scale-invariant quantities and applying the forced $\e$-regularity criterion, see \cite[Corollary 1]{CKN}. Moreover, in \eqref{eq-PfSv-0} we can even take $f(r) = r |\ln r| \cdot \ln \ln |\ln r|$. Since those bring no additional difficulty in our analysis,  we shall focus on the unforced equation and the particular choice of $f$ in this article.
\end{remark}

The key idea is to find a sharp quantitative counterpart for the absolute continuity property of the dissipation energy $\int |\nabla v|^2 dxdt$. For any integrable function $g$ defined on an open set $\Omega \subset \R^n$ and any sequence of open sets $U_k \subset \Omega$,  consider the quantities
\begin{equation}
\EE_k = \int_{U_k \setminus U_{k+1}} |g| dx.
\end{equation}
Clearly, $\EE_k$ satisfies the non-overlapping property
\begin{equation}
\sum_{k \ge 1} \EE_k \le \int_{\Omega} |g| dx,
\end{equation}
which enables the application of the pigeonhole principle. In particular, we have a sequence of $k_j \to +\infty$ such that $\EE_{k_j} \le o(\frac{1}{k_j \ln k_j})$. We will take $g = |\nabla v|^2$, and this observation provides a crucial extra smallness leading to the improvement of partial regularity. 

Another ingredient in the proof of Theorem \ref{thm-partial-reg} is to obtain an iteration of the scale-invariant quantities $A(r)$, $E(r)$, $F(r)$ and $D(r)$ (see Section \ref{sec-notations}) based on the energy structure of the Navier-Stokes equations, assuming the smallness of $\E(r) = r^{-\frac{14}{9}} \left( \int_{Q(r)} |\nabla v|^{\frac95} dx dt \right)^\frac{10}{9}$ instead of $E(r) = \frac{1}{r} \int_{Q(r)} |\nabla v|^2 dxdt$ (see Lemmas \ref{lem-multiplicative}, \ref{lem-energy-iter} and \ref{lem-criterion} for details). Note that the quantity $\E$ was already used by Choe and Lewis in \cite{ChoeLewis}. The main novelty here is the iteration argument.

Next, we turn to our second main result concerning quantitative intervals of regularity. It is well-known that one \emph{cannot} expect regularity estimates only in terms of the local kinetic and dissipation  energies due to the supercritical nature of the Navier-Stokes equations.  In the next result, based on the quantitative method described above, we show that such regularity estimates are nevertheless valid in certain \emph{intervals of regularity} for any local suitable weak solution. Let $v$ be a local suitable weak solution in $Q(1)$ and $\pi_3$ be the projection map to the $x_3$-axis. According to Theorem \ref{thm-partial-reg-old}, the 1D Hausdorff measure of $\pi_3(\mathcal{S}[v]) = \{-1<z<1: \R^2 \times\{z\} \times \R \cap \mathcal{S}[v] \neq \emptyset\}$ must be zero.  Note that $\pi_3(\mathcal{S}[v] \cap \overline{Q(\rho)})$ is a closed set for any $\rho<1$.  Hence, for any open interval $(a,b) \subset (-1,1)$, there exists a subinterval $(z_1, z_2) \subset (a,b)$ such that $v$ is regular in the strip $\{(x,t) \in Q(\rho): z_1<x_3<z_2\}$.  However, Theorem \ref{thm-partial-reg-old} does not imply any explicit lower bound on the length of $(z_1, z_2)$ or any regularity estimates in such a strip. Theorem  \ref{thm-interval} provides a quantitative description of such intervals only in terms of the natural local energies, which will be important for applications in the axially symmetric case (Theorems \ref{thm-small-Gamma} and \ref{thm-improved-ckn}). 

\begin{mainthm}[quantitative interval of regularity] \label{thm-interval}
Let $v$ be a  local suitable weak solution to the Navier-Stokes equations in $Q(1)$. For any $0<\rho<1$ and any $-1<a<b<1$, there exists an interval $(z_1, z_2) \subset (a, b)$, with
$$|z_2-z_1| \ge h_1(\G, \H) = (\G+2)^{-M_1(\min\{\H,\G\} + 1) },$$ 
for a constant $M_1(\rho,a,b)>0$, and $v$ is regular in the closure of the domain $\mathcal{I} = \{(x,t) \in Q(\rho): z_1<x_3<z_2 \}$ and satisfies the estimates
\begin{equation} \label{eq-00114}
\|\nabla_x^m v\|_{L^\infty(\overline{\mathcal{I}})} \lesssim_m h_1^{-(m+1)}, \quad \forall \  m = 0, 1, 2, \cdots.
\end{equation}
\end{mainthm}

\begin{remark}
With simple modifications in the geometries of the proof, we can also obtain the existence of quantitative annulus of regularity. Precisely, for any  $0<\rho<1$ and $0<a<b<1$, there exists an interval $(r_1,r_2) \subset (a,b)$ with
$$|r_2 - r_1| \ge \widetilde{h_1}(\G, \H) =  (\G+2)^{-\widetilde{M_1}(\min\{\H,\G\} + 1) }$$
where $\widetilde{M_1} = \widetilde{M_1}(\rho, a,b) > 0$, such that $v$ is regular  in the closure of the domain $\widetilde{\mathcal{I}} = \{(x,t): r_1<|x|<r_2, -\rho^2<t<0 \}$, with estimates similar to \eqref{eq-00114}.
\end{remark}

For comparison and also for applications in the axially symmetric case, we present a similar result concerning the existence of regular epochs (intervals in the time direction), which is less surprising and much easier to prove.
Theorem \ref{thm-partial-reg-old} implies that the $\frac12$D Hausdorff measure of the set $\pi_0(\mathcal{S}[v]) = \{-1<t\le 0: \(\R^3 \times \{t\}\) \cap \mathcal{S}[v] \neq \emptyset\}$ is zero, and it turns out one can obtain Proposition \ref{prop-epoch}  without using the pigeonhole principle for the hollowed dissipation energies $\EE_k$.

\begin{proposition}[quantitative epoch of regularity] \label{prop-epoch}
Let $v$ be a  local suitable weak solution to the Navier-Stokes equation in $Q(1)$. For any $0<r<1$ and $-1<a,b<0$,  there exists an interval $(t_1, t_2) \subset (a,b)$ with
$$|t_2-t_1| \ge h_2(\G) = M_2(\G+1)^{-M_3},$$ 
for a constant $M_2(r,a,b)>0$  and an absolute constant $M_3>0$, and $v$ is regular in the closure of the domain $\mathcal{J} = \{(x,t) \in Q(r): t_1 < t < t_2\}$ and satisfies the estimates
\begin{equation}
\|\nabla_x^m v\|_{L^\infty({\overline{\mathcal{J}}})} \lesssim_m h_2^{-(m+1)}, \quad \forall \  m = 0, 1, 2, \cdots.
\end{equation}
\end{proposition}

We mention the recent interesting work of Tao \cite{Tao}, where a quantitative triple-exponential improvement of the $L_t^\infty L_x^3$ regularity criterion is proved. The idea is to replace the compactness arguments and unique continuation methods in \cite{Lin} by suitable quantitative counterparts.  See also the works \cite{BarkerPrange, Palasek, Barkernew} for further developments in this direction. Note that certain quantitative annuli and epochs of regularity are proved for $L^\infty_tL^3_x$ global-in-space solutions, and played important roles in \cite{Tao}.

\subsubsection{Applications in the axially symmetric case}

In cylindrical coordinates $(\varrho, \theta, z)$, we have the orthonormal frame
\begin{equation}
e_\varrho = \( \begin{matrix}
\cos \theta \\
\sin \theta \\
0
\end{matrix}\), \quad e_\theta = \( \begin{matrix}
-\sin \theta \\
\cos \theta \\
0
\end{matrix}\), \quad e_z = \( \begin{matrix}
0 \\
0 \\
1
\end{matrix}\),
\end{equation}
under which we can write $v = v^\varrho e_\varrho + v^\theta e_\theta + v^z e_z$. We say that $(v,p)$ is axially symmetric, if the functions $v^\varrho, v^\theta, v^z, p$ are independent of $\theta \in [0,2\pi)$. An important feature of the axially symmetric Navier-Stokes equations (ASNS) is that the scalar function $\Gamma = \varrho v^\theta$ satisfies the transport-diffusion equation
\begin{equation} \label{eq-Gamma-intro}
\partial_t \Gamma - \Delta \Gamma + \frac{2}{\varrho} \partial_\varrho \Gamma + b \cdot \nabla \Gamma = 0
\end{equation}
where $b = v^\varrho e_\varrho + v^z e_z$. The maximum principle for  \eqref{eq-Gamma-intro} gives the scale-invariant a priori bound for mild solutions to the Cauchy problem:
\begin{equation} \label{eq-Gamma-apr}
\|\Gamma(t)\|_{L^\infty(\R^3)} \le \|\Gamma(0)\|_{L^\infty(\R^3)}.
\end{equation}
Also, by the technique of Moser iteration, $\Gamma$ is locally bounded for any axially symmetric suitable weak solution \cite{Seregin0}. 

In \cite{LeiZhang}, the first author together with Zhang proved global regularity for mild solutions under the assumption that $\Gamma$ enjoys a rather weak modulus of continuity $|\Gamma| \le C_* |\ln \varrho|^{-2}$ near the $z$-axis. The proof indicates that the global regularity of ASNS is of a critical nature.  In \cite{Wei}, Wei weakened the continuity assumption to $|\Gamma| \le |\ln \varrho|^{-\frac32}$.  Moreover, in \cite{LeiZhang, Wei}, for the Cauchy problem  it is shown that when $\|\Gamma\|_{L^\infty}$ is small respect to the initial data,  the mild solution stays regular for all times. However, we emphasize that the latter small-swirl regularity results rely heavily on the \emph{global-in-space} energy estimates for the unknowns $(J, \Omega) = \(\frac{\omega^\varrho}{\varrho}, \frac{\omega^\theta}{\varrho}\)$ where $\omega = \omega^\varrho e_\varrho + \omega^\theta e_\theta + \omega^z e_z$ is the vorticity, which do not allow the use of cut-off functions in the $z$-direction. See also \cite{LiuYanlin} for a related study for critical initial data.

In Section \ref{sec-4}, as an application of Theorem \ref{thm-interval}, we prove a quantitative local small-swirl regularity criterion for ASNS. Note that it is not so difficult to prove a non-quantitative version (see Proposition \ref{prop-small-Gamma}) based on contradiction and compactness arguments. However, the quantification of the proof is a nontrivial task.  Using Theorem \ref{thm-interval}, we resolve the difficulties in localizing the energy estimates for $(J, \Omega)$ and thus obtain:

\begin{mainthm}[local small-swirl regularity] \label{thm-small-Gamma}
Let $v$ be an axially symmetric local suitable weak solution to the Navier-Stokes equation in $Q(1)$. There exists an absolute constant $M_4>0$ such that, if
\begin{equation} \label{eq-Gamma-condition}
\|\Gamma\|_{L^\infty(Q(1))} \le M_4^{-1} \big[(\min\{\H, \G\}+1) \ln(\G+2)\big]^{-\frac32}, 
\end{equation} 
then $v$ is regular in $\overline{Q(\frac12)}$.
\end{mainthm}

 It is desirable to replace the condition \eqref{eq-Gamma-condition} by an absolute smallness assumption $\|\Gamma\|_{L^\infty(Q(1))} \le c$, where $c$ is independent of $v^{\varrho}$ and $v^z$. To the best of our knowledge, this is beyond the reach of existing methods.

An interesting application of Theorem \ref{thm-small-Gamma} is a slightly supercritical improvement of the CKN criterion for ASNS. Although we have improved partial regularity for the general Navier-Stokes equations, we can only improve the CKN criterion in the axially symmetric case.  In the important works \cite{SereginActa, CSYT1, CSTY2}, regularity of ASNS is proved under the pointwise Type I bound 
\begin{equation} \label{eq-pw-upper-bound}
|v| \le \frac{C_*}{\varrho^{\alpha} \sqrt{-t}^{1-\alpha}}, \quad \alpha \in [0,1].
\end{equation}
  In \cite{Seregin0}, the  \emph{one-point} Type I result for suitable weak solutions to ASNS
\begin{equation} \label{eq-seregin-onept}
\limsup_{r \to 0+} \frac{1}{r} \int_{Q(r)} |\nabla v|^2 dxdt    < +\infty \implies (0,0) \notin \mathcal{S}[v].
\end{equation}
is established via compactness arguments and Liouville theorem. There are a few slightly supercritical generalizations of the Type I results. A slight improvement of  \eqref{eq-pw-upper-bound}, $\alpha =0$ case, is given by Pan \cite{Pan}.  By the more recent works \cite{ChenTsaiZhang, Seregin2} (and Lemma \ref{lem-log-grow}), regularity of a local suitable weak solution can be proved if, for a small $\nu>0$,
\begin{equation} \label{eq-ctz}
\sup_{w\in Q(1)} \limsup_{r \to 0+} \frac{\int_{Q(r,w)} |\nabla v|^2 dxdt}{r (\ln |\ln r|)^\nu}     < +\infty.
\end{equation}
Our next result is a corollary of Theorem \ref{thm-small-Gamma}, from which all the (slightly supercritical) Type I results in the literature can be deduced (see Remark \ref{rem-final} for details).

\begin{mainthm}[Improved CKN criterion] \label{thm-improved-ckn}
 There exists an absolute constant $\mu>0$ such that, for any axially symmetric local suitable weak solution $v$  to the Navier-Stokes equations in $Q(1)$,
\begin{equation} \label{eq-improved-ckn}
\limsup_{r \to 0+} \frac{\int_{Q(r)} |\nabla v|^2 dxdt }{r (\ln |\ln r|)^\mu}   < +\infty \implies (0,0) \notin \mathcal{S}[v].
\end{equation}
\end{mainthm}

Theorem \ref{thm-improved-ckn} directly improves Seregin's one-point result \eqref{eq-seregin-onept}. Our condition \eqref{eq-improved-ckn} is much weaker than \eqref{eq-ctz} since the bound is assumed only \emph{at a single point}.  We emphasize that a direct consequence of Theorem \ref{thm-improved-ckn} or \eqref{eq-ctz} reads: there exist  $\kappa, \sigma>0$ such that for all $0 \le \alpha \le 1$, we have
$$|v| \le C_* \frac{ ( \ln |\ln \varrho|)^\kappa \cdot (\ln |\ln (-t)|)^\sigma }{\varrho^\alpha \sqrt{-t}^{1-\alpha}}  \ \mbox{in} \ Q(r) \ \mbox{for some} \ r>0  \implies (0,0) \notin \mathcal{S}[v].$$
This obviously improves the Type I criteria in \cite{SereginActa, CSYT1, CSTY2} and generalizes the result of \cite{Pan}.

Finally, we point out that via standard covering arguments, Theorem \ref{thm-improved-ckn} implies the partial regularity result $\mathcal{P}^f(\S[v]) = 0$ for $f = r(\ln |\ln r|)^\mu$, a statement weaker than Theorem \ref{thm-partial-reg}. It is an interesting question whether \eqref{eq-improved-ckn} still holds true with $(\ln |\ln r|)^\mu$ in the denominator replaced by $|\ln r|$. We also mention that the recent numerical results \cite{Hou} seem to suggest that ASNS may develop singularity in a nearly self-similar scenario where the blow-up rate is Type I with some logarithmic corrections.

\subsection{Organization}

In Section  \ref{sec-pre}, we introduce the notations to be used throughout the paper, and state a few preliminary lemmas on scale-invariant quantities and coverings. In Section \ref{sec-2}, we use a quantitative version of the CKN criterion and the pigeonhole principle to  prove Theorem \ref{thm-partial-reg}. In Section \ref{sec-3}, we prove the quantitative results concerning the intervals and epochs of regularity, \emph{i.e.}, Theorem \ref{thm-interval} and Proposition \ref{prop-epoch}. In Section \ref{sec-4}, for comparison, we first prove a non-quantitative version of Theorem \ref{thm-small-Gamma} using a compactness argument. Then we apply the results of  Section \ref{sec-3} to show Theorem \ref{thm-small-Gamma}. Theorem \ref{thm-improved-ckn} is finally deduced as a corollary of Theorem \ref{thm-small-Gamma}. The proofs of the auxilliary Lemmas \ref{lem-Gamma-bound} \& \ref{lem-est-away} are given in the Appendix.
%\begin{itemize}
%\item Section \ref{sec-pre}: notations, iteration of scale-invariant quantities and covering lemmas.
%\item Section \ref{sec-2}: a quantitative version of the CKN criterion; the proof of Theorem \ref{thm-partial-reg}.
%\item Section \ref{sec-3}: the proofs of Theorem \ref{thm-interval} \& Proposition \ref{prop-epoch}.
%\item Section \ref{sec-4}: the non-quantitative version of Theorem \ref{thm-small-Gamma}; the proofs of Theorems \ref{thm-small-Gamma} \& \ref{thm-improved-ckn}. 
%\item Appendix: the proofs of Lemmas \ref{lem-Gamma-bound} \& \ref{lem-est-away}.
%\end{itemize}

\section{Preliminaries} \label{sec-pre}

\subsection{Notations} \label{sec-notations}
\begin{itemize}

\item We write $c, C$ or $O(1)$ for positive absolute constants whose values may vary from line to line. We write $X \lesssim Y$ for $X \le CY$. $X \sim Y$ stands for $X \lesssim Y$ and $Y \lesssim X$. $O_m(1)$ stands for a positive constant which depends only on $m$.  $X \lesssim_m Y$ means that $X \le O_m(1) \, Y$.

\item For a point $w = (x,t) \in \R^3 \times \R$, we denote
$$B(r, x) = \{y \in \R^3: |x-y|<r\},$$
and
$$Q(r,w) = B(r,x) \times (-r^2 + t, t).$$
For simplicity, we write $Q(r) = Q(r,(0,0))$ and $B(r) = B(r, (0,0))$. We also need $Q^+(r,w) = B(r,x) \times (-r^2+t, r^2+t)$.

\item We shall use the following scale-invariant quantities.
\begin{equation*}
A(r) = \esssup_{t \in (-r^2, 0)} \frac{1}{r} \int_{B(r)} |v(x,t)|^2 dx, \quad E(r) = \frac{1}{r} \int_{Q(r)} |\nabla v|^2 dxdt,
\end{equation*}
\begin{equation*}
F(r) = \frac{1}{r^2} \int_{Q(r)} |v(x,t)|^3 dxdt, \quad D(r) = \frac{1}{r^2} \int_{Q(r)} |p(x,t)|^\frac32 dxdt,
\end{equation*}
\begin{equation}
\E(r) = r^{-\frac{14}{9}} \left( \int_{Q(r)} |\nabla v|^{\frac95} dx dt \right)^{\frac{10}{9}}.
\end{equation}
Using H\"older's inequality, it is clear that 
\begin{equation} \label{eq-tildeE-E}
\E(r) \le \(\frac43 \pi\)^{\frac19} E(r).
\end{equation}
$A(r,w), E(r,w)$, \emph{etc.} will denote the corresponding scale-invariant quantities for $Q(r,w)$ instead of $Q(r)$. Moreover, we introduce
$$E^+(r,w) =  \frac{1}{r} \int_{Q^+(r,w) \cap Q(1)} |\nabla v|^2 dxdt.$$

\item We use the usual notations $L^s_t L^l_x$ for the spacetime mixed Lebesgue spaces, \emph{e.g.},
\begin{equation}
\|v\|_{L^\infty_t L^2_x(Q(1))} = \esssup_{-1<t<0} \(\int |v|^2 dx \)^\frac12.
\end{equation}
For simplicity, we write $L^s(Q(1)) = L_t^s L_x^s(Q(1))$.

\end{itemize}

\subsection{Iteration of scale-invariant quantities} \label{sec-iter}

We start with the iterations of scale-invariant quantities similar to those in \cite{Lin, LadySer}. We shall see that in many of such estimates,  the scale-invariant quantity $E$ which plays an important role can be replaced by $\E$.

\begin{lemma} \label{lem-multiplicative}
For any $0<r \le \rho$, we have the multiplicative estimates
\begin{equation} \label{eq-Fr}
F(r) \lesssim \( \frac{r}{\rho} \)^3 \min \left\{ F(\rho), A^{\frac32}(\rho) \right\} + \( \frac{\rho}{r} \)^2  \min \left\{ A^{\frac35}(\rho) \E^{\frac{9}{10}}(\rho), A^\frac12(\rho) E(\rho) \right\},
\end{equation}
and
\begin{equation} \label{eq-Dr}
D(r) \lesssim \frac{r}{\rho} D(\rho) + \(\frac{\rho}{r}\)^2 \min \left\{ A^{\frac35}(\rho) \E^{\frac{9}{10}}(\rho), A^\frac12(\rho) E(\rho) \right\}.
\end{equation}
\end{lemma}
\begin{proof} 
We only prove the estimates involving $\E$, since the estimates involving $E$ are more standard and follow from the same method. Let $[v]_\rho$ be the mean value of $v$ on $B(\rho)$, then by H\"older's inequality, for $-\rho^2<t<0$,
\begin{align}
[v]_\rho(t)   \lesssim  \frac{1}{\rho^{\frac32}} \(\int_{B(\rho)} |v|^2 dx\)^{\frac12} \le \frac{1}{\rho} A^{\frac12}(\rho),
\end{align}
and
\begin{align}
[v]_\rho(t)   \lesssim  \frac{1}{\rho} \(\int_{B(\rho)} |v|^3 dx\)^{\frac13} 
\end{align}
Hence, we have
\begin{align}
\frac{1}{r^2} \int_{Q(r)} |[v]_\rho|^3 dxdt \lesssim  \( \frac{r}{\rho} \)^3 \min \{ F(\rho), A^{\frac32}(\rho) \}.
\end{align}
We estimate $F(r)$ using interpolation and the Sobolev embedding $W^{1,{\frac95}}_x \hookrightarrow L^{\frac92}_x$,
\begin{align} \label{eq-Fr-interp}
F(r) &= \frac{1}{r^2} \int_{Q(r)} |v|^3 dxdt \nonumber \\
&\lesssim \frac{1}{r^2} \int_{Q(r)} |[v]_\rho|^3 dxdt  + \frac{1}{r^2} \int_{Q(r)} |v - [v]_\rho|^3 dxdt \nonumber\\
&\lesssim  \( \frac{r}{\rho} \)^3 \min \{ F(\rho), A^{\frac32}(\rho) \}  +  \frac{1}{r^2} \|v-[v]_\rho\|_{L^\infty_t L^2_x(Q(\rho))}^{\frac65} \|v-[v]_\rho\|_{L^{\frac95}_t L^{\frac92}_x(Q(\rho))}^{\frac95} \nonumber \\
&\lesssim  \( \frac{r}{\rho} \)^3 \min \{ F(\rho), A^{\frac32}(\rho) \}  +  \frac{1}{r^2} \|v-[v]_\rho\|_{L^\infty_t L^2_x(Q(\rho))}^{\frac65} \|\nabla v\|_{L^{\frac95}_t L^{\frac95}_x(Q(\rho))}^{\frac95} \nonumber \\
&\lesssim \( \frac{r}{\rho} \)^3 \min \{ F(\rho), A^{\frac32}(\rho) \} + \( \frac{\rho}{r} \)^2  A^{\frac35}(\rho) \E^{\frac{9}{10}}(\rho).
\end{align}
To estimate $D(r)$, we use the well-known decomposition of pressure
\begin{equation}
p = p_1 + p_2,
\end{equation}
with $p_1$ is given by
\begin{equation} 
 p_1 = -\Delta^{-1} \mbox{div} \mbox{div} \( \mathbf{1}_{B(\rho)} (v-[v]_{\rho})\otimes(v-[v]_{\rho})\)
\end{equation}  
for each time $-\rho^2<t<0$. Clearly, $p_2$ is harmonic in space, hence
\begin{align}
\int_{Q(r)} |p_2|^{\frac32} \le \left(\frac{r}{\rho}\right)^3 \int_{Q(\rho)} |p_2|^{\frac32}.
\end{align}
By the theory of singular integrals, $p_1$ satisfies the estimate
\begin{align} \label{eq-220}
\int_{Q(\rho)} |p_1|^{\frac32} \lesssim \int_{Q(\rho)} |v-[v]_{B_\sigma(\rho)}|^3.
\end{align}
As we have shown in \eqref{eq-Fr-interp}, 
\begin{align} \label{eq-220}
\int_{Q(\rho)} |v-[v]_{B_\sigma(\rho)}|^3 \lesssim \rho^2 A^{\frac35}(\rho) \E^{\frac{9}{10}}(\rho).
\end{align}
Hence, combining the above estimates, we have
\begin{align}
D(r) &= \frac{1}{r^2} \int_{Q(r)} |p|^{\frac32} dxdt \nonumber \\
&\lesssim \frac{1}{r^2} \int_{Q(r)} |p_1|^{\frac32} dxdt + \frac{1}{r^2} \int_{Q(r)} |p_2|^{\frac32} dxdt \nonumber \\
&\lesssim  \(\frac{1}{r^2} + \frac{r}{\rho^3}\) \int_{Q(r)} |p_1|^{\frac32} dxdt + \frac{r}{\rho^3} \int_{Q(\rho)} |p|^{\frac32} dxdt \nonumber \\
&\lesssim \(\frac{\rho}{r}\)^2 A^{\frac35}(\rho) \E^{\frac{9}{10}}(\rho) + \frac{r}{\rho} D(\rho).
\end{align}
\end{proof}

%energy inequality iter
\begin{lemma} \label{lem-energy-iter}
Let $(v, p)$ be a local suitable weak solution to the Navier-Stokes equations in $Q(\rho)$, then for $0<2r<\rho$,
\begin{align}
A(r) + E(r) &\lesssim F^{\frac23}(2r) + F^{\frac13}(2r) D^{\frac23}(2r) \nonumber \\
&\quad + \min \left\{ F^{\frac49}(2r)  A^{\frac13}(2r) \E^{\frac12}(2r), \  F^{\frac13}(2r) A^\frac12(2r) E^\frac12(2r), \ F(2r) \right\} .
\end{align}
\end{lemma}

\begin{proof}
We only prove the estimate involving $\E$, since the others are standard. By taking an appropriate test function in the local energy inequality \eqref{eq-local-energy}, and using interpolation,  we have
\begin{align}
A(r) + E(r) &\lesssim \frac{1}{r^3} \int_{Q(2r)} |v|^2 dxdt + \frac{1}{r^2} \int_{Q(2r)} |v|\(|v|^2 - [|v|^2]_{2r} + 2p\) \nonumber\\
& \lesssim  F^{\frac23}(2r) + \frac{1}{r^2} \(\int_{Q(2r)}|v|^3 dxdt\)^{\frac13}  \(\int_{Q(2r)} \(|v|^2 - [|v|^2]_{2r} \)^{\frac32} dxdt\)^{\frac23} + \nonumber\\
&\quad + \frac{1}{r^2} \(\int_{Q(2r)}|v|^3 dxdt\)^{\frac13}  \(\int_{Q(2r)} |p|^{\frac32} dxdt\)^{\frac23} \nonumber\\
&\lesssim  F^{\frac23}(2r) + \frac{F^{\frac13}(2r)}{r^{\frac43}}  \(\int_{-4r^2}^0 \(\int_{B(2r)} |v| |\nabla v| dx\)^{\frac32} dt\)^{\frac23} + \nonumber\\
&\quad + F^{\frac13}(2r) D^{\frac23}(2r) \nonumber\\
&\lesssim  F^{\frac23}(2r) + \frac{F^{\frac13}(2r)}{r^{\frac43}}  \|v\|^{\frac23}_{L_t^\infty L_x^2(Q(2r))} \|v\|^{\frac13}_{L_t^3 L_x^3(Q(2r))} \|\nabla v\|_{L_t^{\frac95}L_x^{\frac95}(Q(2r))} + \nonumber\\
&\quad + F^{\frac13}(2r) D^{\frac23}(2r) \nonumber\\
&\lesssim  F^{\frac23}(2r) + F^{\frac49}(2r)  A^{\frac13}(2r) \E^{\frac12}(2r) + F^{\frac13}(2r) D^{\frac23}(2r). \nonumber
\end{align}
\end{proof}

\begin{lemma} \label{lem-log-grow}
Let $v$ be a local suitable weak solution to the Navier-Stokes equations in $Q(1)$. For a point $w = (x , t) \in Q(1)$, suppose that
\begin{equation} \label{eq-log-grow}
\limsup_{r \to 0+} \frac{E(r,w)}{|\ln r|} \le 1,
\end{equation}
then there exists an absolute constant $C_1>0$ such that
\begin{equation} \label{eq-log-grow-2}
\limsup_{r \to 0+} \frac{(A + F + D)(r, w)}{|\ln r|^2} \le C_1.
\end{equation}
Similarly, suppose that for some $\mu>0$
\begin{equation} \label{eq-loglog-grow}
\limsup_{r \to 0+} \frac{E(r,w)}{(\ln |\ln r|)^\mu} \le 1,
\end{equation}
then
\begin{equation} \label{eq-loglog-grow-2}
\limsup_{r \to 0+} \frac{(A + F + D)(r, w)}{(\ln |\ln r|)^{2\mu}} \le C_1.
\end{equation}
\end{lemma}

\begin{proof}
Without loss of generality, we let $w = (0, 0)$. Consider the logarithmic case \eqref{eq-log-grow} first. Using Lemma \ref{lem-multiplicative}, for a small number $\theta>0$ there holds
\begin{align} \label{eq-F-FA}
F(\theta r) &\le \frac{1}{16} F\(\frac{r}{2}\) + C_\theta A^\frac12\(\frac{r}{2}\) E\(\frac{r}{2}\) \le  \frac14 F(r) + C |\ln r| A^\frac12\(\frac{r}{2}\),
\end{align}
and
\begin{align} \label{eq-D-DA}
D(\theta r) &\le \frac1{16} D\(\frac{r}{2}\) + C_\theta A^\frac12\(\frac{r}{2}\) E\(\frac{r}{2}\) \le \frac14 D(r) + C |\ln r| A^\frac12\(\frac{r}{2}\).
\end{align}
Using Lemma \ref{lem-energy-iter}, we have
\begin{equation} \label{eq-A-FD1}
A\(\frac{r}{2}\)  \le C (F(r)+D(r)+1).
\end{equation}
Combining \eqref{eq-F-FA}--\eqref{eq-A-FD1}, we arrive at
\begin{equation} \label{eq-FD-FDlog}
F(\theta r ) + D(\theta r) \le \frac12 (F(r) + D(r)) + C |\ln r|^2.
\end{equation}
An induction using \eqref{eq-FD-FDlog} gives, for sufficiently small $r$, 
$$F(r ) + D( r) \le C |\ln r|^2,$$
which, together with \eqref{eq-A-FD1},  implies \eqref{eq-log-grow-2}.

The proof for the double logarithmic case \eqref{eq-loglog-grow} $\implies$ \eqref{eq-loglog-grow-2} is similar.
\end{proof}

\subsection{Covering lemmas}

We recall the following version of Vitali covering lemma for parabolic cylinders, see, \emph{e.g.}, \cite[Lemma 6.1]{CKN}.

\begin{lemma} \label{lem-vitali}
Let $\mathfrak{I}$ be  any collection of parabolic cylinders $\{Q^+(r, w)\}$ contained in a bounded subset of $\R_x^n \times \R_t$. Then there exists a finite or countable subset $\mathfrak{I}' \subset \mathfrak{I}$ such that the cylinders in $\mathfrak{I}'$ are mutually disjoint and 
$$\bigcup\limits_{Q^+(r,w) \in \mathfrak{I}} {Q^+(r,w)} \subset \bigcup\limits_{Q^+(r,w) \in \mathfrak{I}'} Q^+(5r, w).$$
\end{lemma}

The following simple variation is useful in Section \ref{sec-3}.

\begin{lemma} \label{cor-vitali}
Let $\mathfrak{I}$ be  any collection of parabolic cylinders $\{Q^+(r, w)\}$ contained in a bounded subset of $\R^n_x \times \R_t$. Then there exists a finite or countable subset $\mathfrak{I}' \subset \mathfrak{I}$ such that the intervals $(x_3(w)-r, x_3(w)+r)$ for $Q^+(r,w) \in \mathfrak{I}'$ are mutually disjoint and 
$$\bigcup\limits_{Q^+(r,w) \in \mathfrak{I}} (x_3(w) - r, x_3(w) + r) \subset \bigcup\limits_{Q^+(r,w) \in \mathfrak{I}'}(x_3 - 5r, x_3 + 5r).$$
In particular, the parabolic cylinders in $\mathfrak{I}'$ are also mutually disjoint.
\end{lemma}

\begin{proof}
This is a consequence of the usual Vitali covering lemma applied to the collection of intervals $\{(x_3(w) - r, x_3(w) + r): Q^+(r,w) \in \mathfrak{I}\}$.
\end{proof}

\section{Improved partial regularity} \label{sec-2}

\begin{lemma} \label{lem-criterion}
There exist positive absolute constants $\e_1, c_1, \beta$ such that the following holds. Let $v$ be a local suitable weak solution to the Navier-Stokes equations in $Q(1)$ with $\G = \G[v,p] = \int_{Q(1)} \(|v|^3 + |p|^{\frac32} \)dxdt < + \infty$ and let $\sigma = \sigma(\G) = c_1 (\G+1)^{-\beta}$. Suppose that for all $ r \in (\sigma, 1)$, we have
\begin{equation} \label{eq-tildeE-assumption}
\E(r) \le \e_1,
\end{equation}
then $v$ is regular in $\overline{Q(\sigma)}$ with the estimates
\begin{equation}
\|\nabla_x^m v\|_{L^\infty(\overline{Q(\sigma)})} \lesssim_m \sigma^{-(m+1)}, \quad m = 0,1,2,\cdots.
\end{equation}
\end{lemma}
\begin{proof}
By \eqref{eq-Fr}, there exists an absolute constant $\e_0 > 0$ such that if 
\begin{equation}
A(r) + E(r) + D^{\frac43}(r) \le \e_0
\end{equation}
then 
\begin{equation}
F(r) + D(r) \le \e_*.
\end{equation}
The latter implies, via Theorem \ref{thm-ereg} and rescaling, that $v$ is regular in $\overline{Q\left(\frac{r}{2}\right)}$ with the estimates
\begin{equation}
 \|\nabla_x^m v\|_{L^\infty\(\overline{Q\left(\frac{r}{2}\right)}\)} \lesssim_m r^{-(m+1)}, \quad m =0,1,2,\cdots.
\end{equation}

Using Lemma \ref{lem-energy-iter}  and the assumption \eqref{eq-tildeE-assumption}, we have, for $0<\theta < \frac14$ and $2\theta^{-1} \sigma \le r \le 1$,
\begin{align}
A(\theta r) + E(\theta r) &\lesssim F^{\frac23}(2\theta r) + F^{\frac49}(2\theta r)  A^{\frac13}(2\theta r) \e_1^{\frac12} + F^{\frac13}(2\theta r) D^{\frac23}(2\theta r) \nonumber \\
&\lesssim F^{\frac23}(2\theta r) + \e_1^{\frac32} A(2\theta r) + D^{\frac43}(2\theta r).
\end{align}
Using Lemma  \ref{lem-multiplicative} and the assumption \eqref{eq-tildeE-assumption}, we have
\begin{equation}
F^{\frac23}(2\theta r) \lesssim \theta^2 A(r) + \theta^{-\frac43} A^{\frac25}(r) \e_1^{\frac35},
\end{equation}
and
\begin{equation}
D^{\frac43}(\theta r) + D^{\frac43}(2\theta r) \lesssim \theta^{4/3} D^{\frac43}(r) + \theta^{-\frac83} A^{\frac45}(r) \e_1^{\frac65}.
\end{equation}
By definition of $A$, it is clear that
\begin{equation}
A(2\theta r) \le \frac{1}{2\theta} A(r).
\end{equation}
Hence, by choosing the absolute numbers $\theta, \e_1$ to be sufficiently small, we have 
\begin{equation} \label{eq-for-iter}
A(\theta r) + E(\theta r) + D^{\frac43}(\theta r) \le \frac{1}{2} \Big( A( r) + D^{\frac43}( r) \Big) + \frac{1}{10} \e_0.
\end{equation}
An iteration using \eqref{eq-for-iter} gives
\begin{equation} \label{eq-iter-result}
\(A + E + D^\frac43\)\(\frac{\theta^k}{2}\)  \le \frac{1}{2^k} \( A +  D^{\frac43}\)\(\frac12\)  + \frac{1}{5} \e_0.
\end{equation}
for $1 \le k \le k_0 \stackrel{\Delta}= \left\lfloor \frac{\ln (4\sigma )}{\ln \theta}\right\rfloor$. By the local energy inequality \eqref{eq-local-energy}  and the definition of $\G$, we deduce that
\begin{equation}  \label{eq-energies-G}
 \int_{Q(\frac{9}{10})} |\nabla v|^2 dxdt  + \esssup_{-\frac{81}{100}<t<0} \int_{B(\frac{9}{10})} |v|^2 dx \lesssim \G + \G^{\frac23} \lesssim \G+1,
\end{equation}
and consequently,
\begin{equation}
A\(\frac12\) + D^{\frac43}\(\frac12\) \lesssim  (\G+1)^{\frac43}.
\end{equation}
Taking $k=k_0$ in \eqref{eq-iter-result}, we have
\begin{align} \label{eq-iter-result-2}
\(A + E + D^\frac43\)\(\frac{\theta^{k_0}}{2}\) &\le 2^{-\frac{\ln(4\sigma)}{\ln \theta}+1} \( A + D^{\frac43} \) \(\frac12\) + \frac{1}{5} \e_0 \nonumber \\
& \le C_3 c_1^{-\frac{\ln 2}{\ln \theta}} (\G + 1)^{\frac{\beta \ln 2}{\ln \theta} + \frac43}  +   \frac{1}{5} \e_0 \nonumber \\
& \le \frac{1}{5} \e_0 + \frac{1}{5} \e_0 < \e_0.
\end{align}
For the last line, we have chosen $c_1$ small and $\beta$ large, so that
$$5 C_3 c_1^{-\frac{\ln 2}{\ln \theta}} < \e_0, \quad \frac{\beta \ln 2}{\ln \theta} + \frac43 \le 0. $$
With the remark at the beginning of the proof, we deduce that 
\begin{equation} \label{eq-227}
 \|\nabla_x^m v\|_{L^{\infty}\(\overline{Q(\theta^{k_0}/4)}\)} \lesssim_m (\theta^{k_0})^{-m-1}, \quad m \ge 0.
\end{equation}
The conclusion of the lemma follows from \eqref{eq-227}, since $\theta^{k_0} \ge 4\sigma$.
\end{proof}

%The next lemma is simple but crucial, providing the extra smallness for our improvement of partial regularity.
%
%\begin{lemma}\label{lem-pigeonhole}
%Let $\{a_k\}_{k \ge 1}$ be any sequence of positive numbers. For an integrable function $f$ on a domain $\Omega \subset \R^n$, define
%\begin{equation} \label{eq-def-EEk-00}
%\EE_k = \sup_{U \in \mathfrak{U}_k} \( \inf_{V \in \mathfrak{U}_{k+1}} \int_{U \setminus V} |f| dx\),
%\end{equation}
%where
%\begin{equation}
%\mathfrak{U}_k = \left\{U \subset \Omega: U \mathrm{ \ is \ open, \ and \ } |U| \le a_k \right\}.
%\end{equation}
%Then we have 
%\begin{equation}
%\sum_{k \ge 1} \EE_k \le \int_{\Omega} |f| dx.
%\end{equation}
%\end{lemma}
%\begin{proof}
%By the definition of $\EE_k$, for any $\e>0$ and each $k \ge 1$, there exists a set $U_k \in \mathfrak{U}_k$, such that for any $V \in \mathfrak{U}_{k+1}$, we have
%\begin{equation} \label{eq-Uk-V-00}
%\int_{U_k \setminus V} |f| dx \ge \EE_k   - 2^{-k}  \e.
%\end{equation}
%By taking $V = U_{k+1}$, and adding  \eqref{eq-Uk-V-00} for all $k \ge 1$, we get
%\begin{align}
%\sum_{k\ge 1} \EE_k &\le \e + \sum_{k \ge 1}\int_{U_k \setminus U_{k+1}} |f| dx \nonumber \\
%&\le \e + \int_{\Omega} |f| dx.
%\end{align}
%Since $\e>0$ is arbitrary, we obtain that
%\begin{align} \label{eq-sum-EEk-00}
%\sum_{k\ge 1} \EE_k  \le \int_{\Omega} |f| dx.
%\end{align}
%\end{proof}

Now, we are ready to give

\begin{proof}[Proof of Theorem \ref{thm-partial-reg}]

Consider the disjoint sets
\begin{equation} \label{eq-S1def}
\mathcal{S}_{1} = \left\{ w \in \mathcal{S}: \limsup_{r \to 0+} \frac{E^+(r, w)}{|\ln r|} \le 1 \right\} 
\end{equation}
and
\begin{equation} \label{eq-S2def}
\mathcal{S}_{2} = \left\{ w \in \mathcal{S}: \limsup_{r \to 0+} \frac{E^+(r, w)}{|\ln r|} > 1 \right\}.
\end{equation}
Clearly, $\mathcal{S} = \mathcal{S}_1 \cup \mathcal{S}_2$. By Lemma \ref{lem-log-grow}, we have $\mathcal{S}_1 = \bigcup\limits_{m\ge 1}\mathcal{S}_{1m}$,
\begin{equation} \label{eq-S1mdef}
\mathcal{S}_{1m} = \left\{ w \in \mathcal{S}_1: \frac{E^+(r,w)}{|\ln r|} \le 2, \  \frac{(A + F +  D)(r,w)}{|\ln r|^2} \le 2C_1 \ \mbox{for all} \  0<r \le m^{-1} \right\},
\end{equation}
where $C_1$ is the constant from Lemma \ref{lem-log-grow}.

Using the standard covering argument, one can deduce that $\mathcal{P}^f(\mathcal{S}_2) = 0$. Indeed, for any $0<\delta <\frac{1}{10}$ and $w \in \mathcal{S}_2$, there exists $0<r_\delta(w)<\delta$ such that  $ Q(r_\delta(w),w) \subset Q(1)$ and $E^+(r_\delta(w),w) > |\ln r_\delta(w)|$. By Lemma \ref{lem-vitali}, there exists a finite or countable set $\mathcal{T} \subset \mathcal{S}_2$ such that $Q^+(r_\delta(w), w)$ are mutually disjoint for $w \in \mathcal{T}$, and
$$\bigcup\limits_{w \in \mathcal{S}_2} Q^+(r_\delta(w),w) \subset \bigcup\limits_{w \in \mathcal{T}} Q^+(5r_\delta(w), w).$$
Consequently, for $f(r) = r |\ln r|$,
\begin{align} \label{eq-s2-measure}
\mathcal{P}^f(\mathcal{S}_2) \le \mathcal{P}^f_\delta (\mathcal{S}_2) &\lesssim \sum_{w \in \mathcal{T}} (5r_\delta(w))|\ln (5r_\delta(w))| \nonumber\\
&\lesssim \sum_{w \in \mathcal{T}} r_\delta(w)|\ln r_\delta(w)| \nonumber\\
& < \sum_{w \in \mathcal{T}} \int_{Q^+(r_\delta(w),w) \cap Q(1)} |\nabla w|^2 dxdt \nonumber\\
& =  \int_{\bigcup\limits_{w \in \mathcal{T}} Q^+(r_\delta(w), w) \cap Q(1)} |\nabla w|^2 dxdt.
\end{align}
The spacetime  volume of the set $\bigcup\limits_{w \in \mathcal{T}} Q^+(r_\delta(w), w) \cap Q(1)$ is bounded by
$$C \sum_{w\in \T} r_\delta(w)^5 \le C \delta^4 \sum_{w\in \T} r_\delta(w) \le C\delta^4 \int_{Q(1)} |\nabla w|^2 dxdt,$$
which tends to 0 as $\delta \to 0$. Hence, by the absolute continuity of integrals, the last line of \eqref{eq-s2-measure} tends to 0 as $\delta \to 0$. This proves $\mathcal{P}^f(\mathcal{S}_2) = 0$.

Now, it suffices to prove that $\mathcal{P}^f(\mathcal{S}_{1m}) = 0$ for each $m = 1,2,3, \cdots$. Let $m$ be fixed. We denote $\SS = \S_{1m} \cap \overline{Q(\frac12)}$ and $R_k = k^{-\gamma k}$, for $k \ge 1$, where $\gamma \ge 2$ is an absolute constant to be determined through the proof. Certainly, there exists $k_0$ such that $R_{k_0} < \frac{1}{100m}$.  For any $w \in  \SS$ and $k \ge k_0$, by Lemma \ref{lem-criterion} (with rescaling and translation of the domain), there exists a number $r_k(w)$ satisfying $ c |\ln R_k|^{-2\beta} R_k \le r_k(w) \le R_k$ where $c$ is an absolute positive constant, and
\begin{equation} \label{eq-tildeE-lower}
\E(r_k(w),w) > \e_1.
\end{equation}
For simplicity, we write $Q_k(w) = Q(r_k(w), w)$, $Q_k^+(w) = Q^+(r_k(w),w)$ and $\lambda Q_k^+(w) = Q^+(\lambda r_k(w), w)$. Clearly, $\frac{R_{k+1}}{R_k} \le k^{-\gamma}$, and by choosing $\gamma$ large we have $c |\ln R_k|^{-2\beta} R_k > R_{k+1}$. By Lemma \ref{lem-vitali}, there exists a finite or countable subset $ \T_k \subset \SS $ such that the spacetime cylinders $Q_k^+(w)$ are mutually disjoint for $w \in \T_k$, and
\begin{equation} \label{eq-covering}
\SS \subset \bigcup\limits_{w \in \SS} Q_k^+(w) \subset \bigcup\limits_{w \in \T_k} 5Q_k^+(w).
\end{equation}
By \eqref{eq-tildeE-lower} and \eqref{eq-tildeE-E}, 
\begin{equation} \label{eq-E-lower}
E(r_k(w), w) \gtrsim \e_1. 
\end{equation}
Denote $U_{k} = \bigcup\limits_{w \in \T_k} Q_k(w) \subset Q(\frac23)$, then  for any $w\in \T_k$, writing 
$$\Lambda(k,w) = \{w' \in \T_{k+1}: Q_{k+1}(w') \subset 2Q_k^+(w)\},$$ 
we have
\begin{align} \label{eq-0236-2}
\left|U_{k+1} \cap Q_k(w) \right| &\lesssim \sum_{w'\in \Lambda(k,w)} r_{k+1}^5(w') \nonumber \\
&\le R_{k+1}^4 \sum_{w'\in \Lambda(k,w)} r_{k+1}(w') \nonumber \\
&\stackrel{\eqref{eq-E-lower}}\lesssim  R_{k+1}^4 \sum_{w'\in \Lambda(k,w)} \int_{Q_{k+1}(w')} |\nabla v|^2 dxdt  \nonumber \\
&\le R_{k+1}^4 \int_{2Q_k^+(w) \cap Q(1)} |\nabla v|^2 dxdt \nonumber \\
&\stackrel{\eqref{eq-S1mdef}}\lesssim  R_{k+1}^4 r_k(w) |\ln r_k(w)|.
\end{align}
We let
\begin{equation} \label{eq-EEk-def1}
\EE_k = \int_{U_k \setminus U_{k+1}}  |\nabla v|^2 dxdt,
\end{equation}
then clearly 
\begin{equation} \label{eq-nonover1}
\sum_{k \ge 1} \EE_k \le \int_{U_1} |\nabla v|^2 dxdt < + \infty.
\end{equation}
Hence, there exists a sequence of $k_j \to + \infty$ such that
\begin{equation} \label{eq-U-V-small} 
\EE_{k_j} = o\Big(\frac{1}{k_j \ln k_j}\Big).
\end{equation}
From \eqref{eq-tildeE-lower}, using H\"older's inequality we have
\begin{align} \label{eq-erk-twoparts}
&\quad \ \e_1 r_k(w)   \nonumber \\
&< \frac{1}{r_k(w)^{\frac59}} \left( \int_{Q_k(w)} |\nabla v|^{\frac95} dx dt \right)^{\frac{10}{9}} \nonumber \\
&= \frac{1}{r_k(w)^{\frac59}} \left( \int_{Q_k(w) \cap U_{k+1}} |\nabla v|^{\frac95} dx dt + \int_{Q_k(w) \setminus U_{k+1}} |\nabla v|^{\frac95} dx dt \right)^{\frac{10}{9}} \nonumber \\
&\le \frac{2|U_{k+1} \cap Q_k(w)|^{\frac19}}{r_k(w)^{\frac59}} \int_{Q_k(w) \cap U_{k+1}} |\nabla v|^2 dxdt  \nonumber \\
&\quad + \frac{2|Q_k(w)|^{\frac19}}{r_k(w)^{\frac59}} \int_{Q_k(w) \setminus U_{k+1}} |\nabla v|^2 dxdt \nonumber \\
&\le C  \(\frac{R_{k+1}}{r_k(w)}\)^\frac49  |\ln r_k(w)|^\frac19 \int_{Q_k(w)} |\nabla v|^2 dxdt + C \int_{Q_k(w) \setminus U_{k+1}} |\nabla v|^2 dxdt.
\end{align}
Note that, by choosing $\gamma$ sufficiently large, we can ensure that, as $k \to +\infty$,
\begin{align}
 \(\frac{R_{k+1}}{r_k(w)}\)^\frac49  |\ln r_k(w)|^\frac19 &\le C \(\frac{R_{k+1}}{R_k}\)^\frac49 |\ln R_k|^{\frac{1+8\beta}{9}} \nonumber \\
&\le C k^{-\frac{4\gamma}{9}} (k\ln k)^{\frac{1+8\beta}{9}}= o\(\frac{1}{k^2}\).
\end{align}
Hence, we add up the estimate \eqref{eq-erk-twoparts} for all $w \in \T_k$ and take $k = k_j$ to get 
\begin{align} 
\sum_{w\in \T_{k_j}} r_{k_j}(w) &\lesssim   o\(\frac{1}{k_j^2}\) \int_{U_{k_j}} |\nabla v|^2 dxdt +   \int_{U_{k_j} \setminus U_{{k_j}+1}} |\nabla v|^2 dxdt \stackrel{\eqref{eq-U-V-small}}= o\(\frac{1}{{k_j} \ln {k_j}}\).
\end{align}
 Note that $k \ln k \sim |\ln R_k| \sim |\ln R_{k+1}| \sim |\ln r_k(w)|$. ($A\sim B$ means $A\le CB$ and $B \le CA$.) Hence, as $j \to +\infty$,
\begin{equation}
\sum_{w\in \T_{k_j}} r_{k_j}(w) |\ln r_{k_j}(w)| = o(1).
\end{equation}
In view of \eqref{eq-covering} and the definition of the $f$-Hausdorff measure, we deduce $\mathcal{P}^f (\SS) = 0$ for $f(r) = r |\ln r|$. 

Using rescaling and translation of the domain, it is then easy to deduce $\mathcal{P}^f (\mathcal{S}_{1m} \cap \overline{Q(r)}) = 0$ for any $r<1$. Consequently, $\mathcal{P}^f(\S_{1m}) = 0$ for $m = 1,2,3,\cdots$, and $\mathcal{P}^f(\S_{1}) = 0$. The conclusion of the theorem now follows readily.
\end{proof}

\section{Interval and epoch of regularity} \label{sec-3}

\begin{lemma} \label{lem-criterion-2}
Let $\e_1$ and $\beta$ be the constants from Lemma \ref{lem-criterion}. There exists a positive absolute constant $c_2$ such that the following holds. Let $v$ be a local suitable weak solution to the Navier-Stokes equations in $Q(1)$ with $\G = \int_{Q(1)} \(|v|^3 + |p|^{\frac32} \)dxdt < + \infty$ and
$\H_1 = \int_{Q(\frac23)} |\nabla v|^2  dxdt$. Let $\sigma = \sigma(\G) = c_2 (\G + 1)^{-2\beta}$. Given any $w \in \overline{Q(\frac12)}$ and any $R \in (0, \frac{1}{10})$, if for all $r \in (R, \frac{1}{10})$, 
\begin{equation} 
E(r,w) \le 100\H_1,
\end{equation}
and for all $r \in (\sigma R, R)$, 
\begin{equation} 
\E(r,w) \le \e_1,
\end{equation}
then $v$ is regular in $\overline{Q(\sigma R,w)}$ with the estimates
\begin{equation}
\|\nabla_x^m v\|_{L^\infty\(\overline{Q(\sigma R,w)}\)} \lesssim_m (\sigma R)^{-(m+1)}, \quad m = 0,1,2,\cdots.
\end{equation}
\end{lemma}

\begin{proof}
We omit the dependence of the scale invariant quantities on the center $w$ for simplicity. By \eqref{eq-energies-G}, $\H_1 \lesssim \G+1$. Using Lemma \ref{lem-multiplicative}, we can find an absolute number $\theta>0$ such that, when $r \in (R, \frac{1}{10})$,
\begin{align} 
F(\theta r) &\le \frac1{16} F\(\frac{r}{2}\) + C A^\frac12\(\frac{r}{2}\) E\(\frac{r}{2}\) \le  \frac14 F(r) + C (\G+1) A^\frac12\(\frac{r}{2}\),  \label{eq-F-G+1} \\
D(\theta r) &\le \frac1{16} D\(\frac{r}{2}\) + C A^\frac12\(\frac{r}{2}\) E\(\frac{r}{2}\) \le \frac14 D(r) + C (\G+1) A^\frac12\(\frac{r}{2}\).
\end{align}
By Lemma \ref{lem-energy-iter},
\begin{equation} \label{eq-A-FD1-new}
A\(\frac{r}{2}\)  \lesssim F(r)+D(r)+1.
\end{equation}
Combining  \eqref{eq-F-G+1}--\eqref{eq-A-FD1-new} we have, when $r \in (R, \frac{1}{10})$,
\begin{equation}
F(\theta r)  + D (\theta r) \le  \frac{1}{2} (F(r) +  D(r)) + C(\G+1)^2.
\end{equation}
Clearly, $F\(\frac1{10}) + D(\frac{1}{10}\) \lesssim \G+1 \lesssim (\G+1)^2$. Induction gives $F(r) + D(r) \lesssim (\G+1)^2$ for all $r \in [R, \frac{1}{10}]$.  Now, we can apply Lemma \ref{lem-criterion} in $Q(R)$ to conclude. 
\end{proof}

\begin{proof}[Proof of Theorem \ref{thm-interval}]
We only deal with the case $\rho = \frac12$, $a=-\frac12$, $b = \frac12$, since the general case follows from the same arguments only with some changes in the constants.  By Theorem \ref{thm-ereg}, when $\G < \e_*$, the conclusion of Theorem \ref{thm-interval} clearly holds. Hence we may assume that $\G > \e_*$, and in particular, $\G \sim \G+1$. Let $\mathcal{H}_1 = \int_{Q(\frac23)} |\nabla v|^2 dxdt$ and 
$$R_k = (\G + 2)^{-\gamma k}, \ k \ge 1$$ 
with $\gamma > 4$ an absolute constant to be determined later. Let $\e_1$ be the absolute constant from Lemma \ref{lem-criterion}. 

Let $\sigma = c_2( \G +1 )^{-2 \beta}$. In view of Lemma \ref{lem-criterion-2}, to prove the theorem, it suffices to show that there exists $z \in (-\frac13, \frac13)$ and $k \in \{1,2,\cdots,\lfloor 100 \e_1^{-1} \mathcal{H}_1 \rfloor+1\}$, such that for all  $|x_1|<\frac12$, $|x_2| < \frac12$ and $-\frac14<t \le 0$, the following two statements hold:
\begin{enumerate}
\item[(a)] for all $r \in (R_k, \frac1{10})$,  $E(r,(x_1,x_2,z,t)) \le 100\H_1$,
\item[(b)] for all $r \in (\sigma R_k, R_k)$, $\E(r, (x_1,x_2,z,t)) \le \e_1$.
\end{enumerate}

Next, we argue by contradiction. Assume that for any $z \in (-\frac13, \frac13)$ and any $k \in \{1,2,\cdots,\lfloor 100 \e_1^{-1} \mathcal{H}_1 \rfloor+1\}$, there exists a point $w_k(z) = (x_{1k}(z), x_{2k}(z), z, t_k(z))$ with  $|x_{1k}(z)|<\frac12$, $|x_{2k}(z)| < \frac12$ and $-\frac14<t_k(z) \le 0$ such that (a) or (b) fails. For any such $k$, we divide the $z$-interval into two disjoint subsets $\(-\frac13, \frac13\) = \Z_{1,k} \cup \Z_{2,k}$ with
\begin{equation} \label{eq-Zk1}
\Z_{1,k} = \left\{ z \in \(-\frac13, \frac13\): \exists  r_{1,k}(z) \in \(\sigma R_k, \frac{1}{10}\), \ E^+(r_{1,k}(z),w_k(z)) \ge 100\H_1\right\}
\end{equation}
and
$$\Z_{2,k} = \(-\frac13, \frac13\) \setminus \Z_{1,k}.$$
By assumption, if $z \in \Z_{2,k}$, then
\begin{equation} \label{eq-E-upper}
E^+(r,w_k(z))<100\H_1
\end{equation}
for any $r \in (\sigma R_k, \frac{1}{10})$, and there exists $r_{2,k}(z) \in (\sigma R_k, R_k)$ such that 
\begin{equation} \label{eq-tildeE-lower-2}
\E(r_{2,k}(z), w_k(z))> \e_1.
\end{equation}
By \eqref{eq-tildeE-lower-2} and \eqref{eq-tildeE-E},
\begin{equation} \label{eq-E-lower-2}
E(r_{2,k}(z), w_k(z)) \gtrsim \e_1.
\end{equation}
By Lemma \ref{cor-vitali}, we can find finite or countable subsets $\T_{1,k} \subset \Z_{1,k}$ and $\T_{2,k} \subset \Z_{2,k}$ such that all the $z$-intervals $(z-r_{1,k}(z), z + r_{1,k}(z)), \ z\in \T_{1,k}$ and $(z-r_{2,k}(z), z + r_{2,k}(z)), \ z\in \T_{2,k}$ are mutually disjoint, and
\begin{equation} \label{eq-z-cover}
\( -\frac13, \frac13 \) \subset \bigcup\limits_{z \in \T_{i,k}, \, i=1,2} (z-5r_{i,k}(z), z + 5r_{i,k}(z)).
\end{equation}
Denote $Q_k(z) =  Q(r_{2,k}(z), w_k(z))$, $\lambda Q_k^+(z) =  Q^+(\lambda r_{2,k}(z), w_k(z))$ and $U_k = \bigcup\limits_{z \in \T_{2,k}} Q_k(z)$.  For any $z\in \T_{2,k}$, writing 
$$\Lambda(k,z) = \{z' \in \T_{2,k+1}: Q_{k+1}(z') \subset 2 Q_k^+(z)\},$$ 
we obtain
\begin{align} \label{eq-0236-2}
\left|U_{k+1} \cap Q_k(z) \right| &\lesssim \sum_{z'\in \Lambda(k,z)} r_{2,k+1}^5(z') \nonumber \\
&\le R_{k+1}^4 \sum_{z'\in \Lambda(k,z)} r_{2,k+1}(z') \nonumber \\
&\stackrel{\eqref{eq-E-lower-2}}\lesssim R_{k+1}^4 \sum_{z'\in \Lambda(k,w)} \int_{Q_{k+1}(z')} |\nabla v|^2 dxdt  \nonumber \\
&\le R_{k+1}^4 \int_{2Q_k^+(z)} |\nabla v|^2 dxdt \nonumber \\
&\stackrel{\eqref{eq-E-upper}}\lesssim R_{k+1}^4 r_{2,k}(z) \H_1.
\end{align}
Let
\begin{equation} \label{eq-EEk-def2}
\EE_k = \int_{U_k \setminus U_{k+1}}  |\nabla v|^2 dxdt,
\end{equation}
then clearly 
\begin{equation}  \label{eq-nonover2}
\sum_{k \ge 1} \EE_k \le \int_{U_1} |\nabla v|^2 dxdt \le \H_1.
\end{equation}
By the pigeonhole principle, there exists $K \in \{1,2,\cdots,\lfloor 100 \e_1^{-1} \mathcal{H}_1 \rfloor+1\}$ such that
\begin{equation} \label{eq-EEk-small}
\EE_{K} \le \frac{\H_1}{\lfloor 100 \e_1^{-1} \mathcal{H}_1 \rfloor+1} \le \frac{\e_1}{100}.
\end{equation}
Using \eqref{eq-Zk1} and the disjointness of $Q^+(r_{1,k}(z), w(z))$, we can estimate
\begin{align} \label{eq-r1k-1100}
\sum_{z\in \T_{1,k}} r_{1,k}(z) &\le \sum_{z\in \T_{1,k}} \frac{1}{100 \H_1} \int_{Q^+(r_{1,k}(z), w(z)) \cap Q(1)} |\nabla v|^2 dxdt \nonumber \\
&\le \frac{1}{100 \H_1} \int_{Q(\frac23)} |\nabla v|^2 dxdt = \frac{1}{100}.
\end{align}
Using \eqref{eq-tildeE-lower-2}, the disjointness of  $Q_k(z)$ and H\"older's inequality, we have
\begin{align}
&\quad \ \sum_{z\in \T_{2,k}} r_{2,k}(z) \nonumber \\
&\stackrel{\eqref{eq-tildeE-lower-2}}\le \sum_{z\in \T_{2,k}} \frac{1}{\e_1 r_{2,k}(z)^{\frac59}} \( \int_{Q_k(z)} |\nabla v|^{\frac95} dxdt \)^\frac{10}{9} \nonumber \\
&\le \sum_{z\in \T_{2,k}} \frac{1}{\e_1 r_{2,k}(z)^{\frac59}} \( \int_{Q_k(z) \cap U_{k+1}} |\nabla v|^{\frac95} dxdt +  \int_{Q_k(z) \setminus U_{k+1}} |\nabla v|^{\frac95} dxdt \)^\frac{10}{9} \nonumber \\
&\stackrel{\mbox{\tiny (H\"older)}}\le \sum_{z\in \T_{2,k}} 2\frac{|U_{k+1} \cap Q_k(z)|^\frac19}{\e_1 r_{2,k}(z)^{\frac59}} \int_{Q_k(z) \cap U_{k+1}} |\nabla v|^2 dxdt + \frac{4}{\e_1}\int_{Q_k(z) \setminus U_{k+1}} |\nabla v|^2 dxdt \nonumber \\
&\le \sup_{z \in \T_{2,k}}\frac{2|U_{k+1} \cap Q_k(z)|^\frac19}{\e_1 r_{2,k}(z)^{\frac59}} \int_{Q(\frac23)} |\nabla v|^2 dxdt + \frac{4}{\e_1}\int_{U_k \setminus U_{k+1}} |\nabla v|^2 dxdt.
\end{align}
Applying \eqref{eq-0236-2} and \eqref{eq-EEk-small}, we obtain
\begin{align} \label{eq-r2k-150}
\sum_{z\in \T_{2,k}} r_{2,k}(z) &\le  C_3 \H_1^\frac{10}9 \(\frac{R_{k+1}}{\inf_{z \in \T_{2,k}}r_{2,k}(z)}\)^{\frac49} + \frac{4}{100} \le \frac{1}{20}.
\end{align}
For the last inequality, we used that,  when $\gamma$ is taken sufficiently large,
$$C_3 \H_1^\frac{10}9 \(\frac{R_{k+1}}{\sigma R_k}\)^{\frac49} \le C (\G + 1)^\frac{10}9 \(\frac{R_{k+1}}{(\G+1)^{-2\beta} R_k}\)^{\frac49} \le C (\G+1)^{\frac{8\beta + 10}{9}} (\G+2)^{-\frac{4\gamma}{9}} \le \frac{1}{100}.$$
On the other hand, by \eqref{eq-z-cover} we have
\begin{equation} \label{eq-23-r1kr2k}
\frac23 \le 10 \( \sum_{z\in \T_{1,k}} r_{1,k}(z) +  \sum_{z\in \T_{2,k}} r_{2,k}(z) \).
\end{equation}
The estimates \eqref{eq-r1k-1100}, \eqref{eq-r2k-150} and \eqref{eq-23-r1kr2k} together give a contradiction, which finishes the proof.
\end{proof}

Next, we show the existence of epochs of regularity using the similar idea as above, without invoking the extra smallness of Dirichlet integral on hollowed sets. 

\begin{proof}[Proof of Proposition \ref{prop-epoch}]
%By Theorem \ref{thm-ereg}, when $\G < \e_1$, the conclusion of Theorem \ref{thm-interval} clearly holds. Hence we assume that $\G > \e_1$. In particular $\G \sim \G+1$.
We only consider the case $\rho= \frac12$, $a=-\frac34$, $b = -\frac12$, since the general case is similar. 

Let $R = \frac{\e_1}{100} \(\int_{Q(\frac{9}{10})} |\nabla v|^2 dxdt + 1\)^{-1} \gtrsim (\G+1)^{-1}$. In view of Lemma \ref{lem-criterion} (with rescaling and translation of the domain), it suffices to show that there exists $t \in (-\frac23, -\frac12)$, such that for any $c_1 (\G R^{-2} + 1)^{-\beta} \le r \le R$ and any $x \in B(\frac12)$, we have  $\E(r,(x,t))<\e_1$. 

We argue by contradiction. Assume that for any $t \in (-\frac23, -\frac12)$, there exists $c_1 (\G R^2 + 1)^{-\beta} \le r(t) \le R$ and $x(t)$ with $|x(t)|<\frac12$, such that $\E(r(t),(x(t),t)) \ge \e_1$. By \eqref{eq-tildeE-E} we also have $E(r(t),(x(t),t)) \ge \frac12 \e_1$. Using the Vitali covering lemma for the $t$-intervals $(t-r^2(t), t), t \in (-\frac23, -\frac12)$, there exists a finite or countable set $\mathcal{T} \subset (-\frac23, -\frac12)$ such that the intervals $(t-r^2(t),t)$ are mutually disjoint for $t \in \mathcal{T}$, and 
\begin{equation}
\(-\frac23, -\frac12\) \subset \bigcup\limits_{t \in (-\frac23, -\frac12)} (t-r^2(t),  t) \subset \bigcup\limits_{t \in (-\frac23, -\frac12)} (t-3r^2(t),  t+2r^2(t)).
\end{equation}
Clearly, $Q(r(t), (x(t), t))$ are also mutually disjoint for $t \in \mathcal{T}$. Consequently, we have
\begin{align}
\frac16 \le \sum_{t \in \mathcal{T}} 5r^2(t) &\le 10 R \sum_{t \in \mathcal{T}} \e_1^{-1} \int_{Q(r(t), (x(t), t))} |\nabla v|^2 dxdt \nonumber \\
&\le 10\e_1^{-1} R \int_{Q(\frac{9}{10})} |\nabla v|^2 dxdt = \frac1{10},
\end{align}
which gives a contradiction. Hence, the lemma is proved.
\end{proof}

Note that a simple alternative proof of Proposition \ref{prop-epoch} can be given using $\e$-regularity Theorem \ref{thm-ereg} and  a pigeonhole principle for the integrals $\int |u|^\frac{10}{3} dxdt$ and $\int \(\int |p|^\frac{5}{3} dx\)^{\frac9{10}} dt$ which have dimensions $\frac53$ and $\frac{17}{10}$ respectively, both less than $2$.

\section{The axially symmetric case} \label{sec-4}

Let $v$ be an axially symmetric (with respect to the $z$-axis) local suitable weak solution in $Q(1)$. Recall that we can use the cylindrical coordinates $(\varrho, \theta,  z)$. Due to Theorem \ref{thm-partial-reg-old} and the axial-symmetry assumption, we know that $v$ is regular away from the $z$-axis and all the singular points must lie on the $z$-axis. Quantitative estimates for $v$ in $\varrho>0$ is presented in Lemma \ref{lem-est-away} below. We point out that the proofs of Theorems \ref{thm-partial-reg} and \ref{thm-interval} can be simplified in the following way. Instead of using the sets $U_k$ to define $\EE_k$ (see \eqref{eq-EEk-def1} and \eqref{eq-EEk-def2}), we can simply introduce the hollowed integrals
\begin{equation}
\EE_k' = \int_{R_{k} <\varrho< R_{k+1}} |\nabla v|^2 dxdt.
\end{equation}
Clearly, we also have the non-overlapping property similar to \eqref{eq-nonover1} and \eqref{eq-nonover2},
$$\sum_{k \ge 1}\EE_k' = \int_{\varrho<R_1} |\nabla v|^2 dxdt < +\infty,$$
which is sufficient for the rest of the proofs.

\subsection{A non-quantitative result} \label{sec-41}

Before embarking on the proof of the quantitative Theorem \ref{thm-small-Gamma}, we present a non-quantitative version of it whose proof relies on a compactness argument. However, we emphasize that a direct quantification of this proof is a nontrivial task. 

\begin{proposition}[non-quantitative small-swirl regularity] \label{prop-small-Gamma}
Let $v$ be an axially symmetric local suitable weak solution in $Q(1)$, and $\Gamma = \varrho v^\theta$. There exists a number $\varepsilon(\G[v,p])>0$, such that if $\|\Gamma\|_{L^\infty(Q(1))} \le \varepsilon$, then $v$ is regular in $Q({\frac12})$.
\end{proposition}
\begin{proof}
We argue by contradiction and follow the stability of singularity argument in \cite[Lemma 2.1]{RusinSverak}. Suppose that the conclusion is wrong, then there exist a sequence of axially symmetric local suitable weak solutions $v_k, p_k$ in $Q(1)$ with
$$\iint_{Q_1} |v_k|^3 + |p_k|^\frac32 \le \G < + \infty,$$
and
$$|\Gamma_k|_{L^\infty} \le \e_k \to 0,$$
and a sequence of points $w_k \in \overline{Q(\frac12)}$ such that $w_k$ is a singular point of $v_k$. Up to taking subsequences we have the convergences (see \cite[Page 367]{LadySer} for the discussion on compactness)
$$w_k \to w_\infty \in \overline{Q\Big(\frac12\Big)},$$
and
$$v_k \to v_\infty \mathrm{\ strongly \ in\ } L_t^{3}L_x^{3}\Big(Q\Big(\frac23\Big)\Big).$$
Note that $v_\infty$ is a suitable weak solution with no swirl ($v_\infty^\theta = 0$). It is known that $v_\infty$ must be regular (see, \emph{e.g.}, \cite{Kang}), and thus
$$\|v_\infty\|_{L^\infty(Q(\frac23))} \le A$$
for some finite number $A$. Let $a(A)> b(A)> 0 $ be small numbers to be chosen later. Using the convergence of $v_k \to v_\infty$, we have, for large $k$,
$$\int_{ B_a(0) \times (-b^2, 0)  + w_\infty} |v_k|^3 dxdt \le 2A a^3 b^2$$
and
$$\int_{Q(b,w_\infty)} |v_k|^3  dxdt < 2A b^5$$
Then, using the usual decomposition of $p_k$ (see the proof of Lemma \ref{lem-multiplicative}) in the domain $B_a(0) \times (-b^2, 0)  + w_\infty$, we have
\begin{equation}
p_k = p_{k1} + p_{k2}
\end{equation}
where $\Delta p_{k1} = 0$ and 
\begin{equation}
\int_{B_a(0) \times (-b^2, 0)  + w_\infty} |p_{k2}|^\frac32 dxdt \lesssim \int_{B_a(0) \times (-b^2, 0)  + w_\infty} |v_k|^3  dxdt
\end{equation}
Consequently, we obtain
\begin{align*}
\int_{Q(b,w_\infty)}|p_k|^\frac32 dxdt &\lesssim \int_{Q(b,w_\infty)}|p_{k1}|^\frac32 dxdt + \int_{Q(b,w_\infty)}|p_{k2}|^\frac32 dxdt \\
&\lesssim \frac{b^3}{a^3} \int_{B_a(0) \times (-b^2, 0)  + w_\infty} |p_{k1}|^\frac32 dxdt + \int_{B_a(0) \times (-b^2, 0)  + w_\infty} |p_{k2}|^\frac32 dxdt \\
&\lesssim \frac{b^3}{a^3} \int_{Q(1)} |p_{k}|^\frac32 dxdt + \int_{B_a(0) \times (-b^2, 0)  + w_\infty} |p_{k2}|^\frac32 dxdt \\
&\lesssim \frac{\G b^3}{a^3} + A a^3 b^2.
\end{align*}
By taking $a(A), b(A)$ suitably small, we get
$$\int_{Q(b,w_\infty)} |v_k|^3 + |p_k|^\frac32 \, dxdt \le \varepsilon_* b^2$$
where $\varepsilon_*$ is from Theorem \ref{thm-ereg}. This contradicts with the assumption that $w_k$ is a singular point of $v_k$. Hence, the proposition is proved.
\end{proof}

\subsection{Quantitative small-swirl regularity}

We need the next two lemmas for general axially symmetric suitable weak solutions.  Lemma \ref{lem-Gamma-bound} is hardly  a   new result. \eqref{eq-Gamma-upper-1} is essentially proved in \cite[Proposition 2.2]{Seregin2}, and \eqref{eq-Gamma-upper-2} can also be obtained using the same idea. Lemma \ref{lem-est-away} concerns quantitative regularity estimates away from the $z$-axis. Following the proof of \cite[Proposition 4.1]{Seregin1}, one would get that $\|\nabla^m v\|_{L^\infty(Q(R)\cap \{\varrho > \delta\})}$ depends exponentially on $\G[v,p]$ for any given $R<1$ and $\delta > 0$. With minor changes in the proof, the dependence can be improved to a polynomial one. For completeness, we give self-contained proofs of Lemmas \ref{lem-Gamma-bound} and \ref{lem-est-away} in the Appendix. 

\begin{lemma} \label{lem-Gamma-bound}
Let $(v,p)$ be a local suitable weak solution to the Navier-Stokes equations in $Q(1)$, and assume that $v, p$ are axially symmetric. Then, $\Gamma = \varrho v^\theta$ satisfies
\begin{equation} \label{eq-Gamma-upper-1}
\|\Gamma\|_{L^\infty(Q(\frac34))} \lesssim (\G[v,p]+1)^{2},
\end{equation}
and
\begin{equation} \label{eq-Gamma-upper-2}
\|\Gamma\|_{L^\infty(Q(\frac34) \cap \{\varrho \ge \frac14\})} \lesssim \G[v,p]+1.
\end{equation}
\end{lemma}

\begin{lemma} \label{lem-est-away}
Let $(v,p)$ be a local suitable weak solution to the Navier-Stokes equations in $Q(1)$, and assume that $v, p$ are axially symmetric. Then for any $R<1$ and any $\delta >0$,
\begin{equation}
\|\nabla^m v\|_{L^\infty(Q(R)\cap \{\varrho > \delta\})} \lesssim_{m,\delta} (\G[v,p]+1)^{O_m(1)}, \quad m =0,1,2,\cdots.
\end{equation}
\end{lemma}

Now, we are ready to present:

\begin{proof} [Proof of Theorem \ref{thm-small-Gamma}]  The main idea is to localize the arguments of Lei-Zhang \cite{LeiZhang} and Wei \cite{Wei} with the help of Theorem \ref{thm-interval} and Proposition \ref{prop-epoch}.

 By Theorem \ref{thm-ereg}, when $\G < \e_1$, the conclusion of Theorem \ref{thm-small-Gamma} clearly holds. Hence we may assume that $\G \ge \e_1$, and in particular, $\G \sim \G+1$. Denote
\begin{equation}
J = \frac{\omega^\varrho}{\varrho} = -\frac{\partial_z v^\theta}{\varrho}
\end{equation}
and recall that for smooth solutions $J$ satisfies the equation
\begin{equation} \label{eq-J}
\partial_t J - \Delta J - \frac{2}{\varrho} \partial_\varrho J + b \cdot \nabla J = \bar{\omega} \cdot \nabla \frac{v^\varrho}{\varrho}. 
\end{equation}
where $\bar{\omega} = \omega^\varrho e_\varrho + \omega^z e_z$ and $\omega = \nabla \times v$ is the vorticity. We point out that $J$ was introduced by Chen-Fang-Zhang in \cite{CFZ}.

\textbf{Step 1}. Using the interval of regularity in Theorem \ref{thm-interval}, we can find $z_1 \in (-\frac23, -\frac12)$ and $z_2 \in (\frac12, \frac23)$ such that $v$ is regular  and satisfies the pointwise estimates
\begin{equation} \label{eq-P1P2}
|\nabla_x^m v| \le (\G+2)^{C (\mathcal{H}+1)} , \quad m = 0,1,2,
\end{equation}
on the hypersurfaces $\mathcal{P}_1 = \{ z = z_1 , \varrho<\frac15, -\frac45 < t \le 0\} \subset Q(1)$ and $\mathcal{P}_2 = \{ z =  z_2, \varrho<\frac15, -\frac45 < t \le 0\} \subset Q(1)$. Using the epoch of regularity in Proposition \ref{prop-epoch}, we can find $t_0 \in (-\frac23, -\frac12)$ such that $v$ is regular and satisfies the pointwise estimates
\begin{equation} \label{eq-P0}
|\nabla_x^m v| \lesssim \G^{C},\quad  m = 0,1,2,
\end{equation}
on the hypersurface $\mathcal{P}_0 = \{t = t_0, |x| < \frac45\} \subset Q(1)$. Denote the domains $\mathcal{D} = \{ z_1 < z < z_2, \varrho<\frac15, t_0<t < 0\}$ and $\mathcal{D}' = \{ z_1 < z < z_2, \varrho<\frac1{10}, t_0<t < 0\}$, and let $\mathcal{D}_t = \mathcal{D} \cap \(\mathbb{R}^3 \times \{t\}\)$, $\mathcal{D}'_t = \mathcal{D}' \cap \(\mathbb{R}^3 \times \{t\}\).$

In the sequel, we shall work in times $t_0<t<T$ where $T$ is the first time for  $v|_\mathcal{D}$ to blow up (we set $T = 0$ if $v$ is regular in $\overline{\mathcal{D}}$).

\textbf{Step 2}. Let $u(\varrho,z,t) = \int_0^\varrho |v^\theta(\varrho',z,t)| d\varrho'$ and $$a(t) = \sup_{(\varrho,z,t) \in \mathcal{D}'}\, \frac{u(\varrho,z,t)}{\varrho}.$$ The function $a$ was introduced by Wei \cite{Wei}. By H\"older's inequality, for $(\varrho,z,t) \in \mathcal{D}'$, we have 
\begin{align}
|u(\varrho,z ,t)|^2 &\le \int_0^\varrho \varrho' d\varrho' \int_0^\varrho \frac{|v^\theta(\varrho', z, t)|^2}{\varrho'} d\varrho' \nonumber \\
&\le \varrho^2 \int_0^\varrho \int_{z_1}^z \frac{|v^\theta(\varrho', z',t)| |\partial_z v^\theta(\varrho', z',t)|}{\varrho'} dz'd\varrho' + \frac{\varrho^2}{2} \int_0^\varrho \frac{|v^\theta(\varrho', z_1, t)|^2}{\varrho'} d\varrho'   \nonumber \\
&\lesssim \varrho^2 \left\|\frac{v^\theta}{\varrho}\right\|_{L^2(\mathcal{D}'_t)} \|J\|_{L^2(\mathcal{D}'_t)} +  \varrho^2  (\G+2)^{C (\mathcal{H}+1)} ,
\end{align}
As a consequence, for $t_0<t<0$, we have
\begin{equation} \label{eq-step2}
a(t)^2 \lesssim \left\|\frac{v^\theta}{\varrho}\right\|_{L^2(\mathcal{D}'_t)} \|J\|_{L^2(\mathcal{D}'_t)} +  (\G+2)^{C(\mathcal{H}+1)} .
\end{equation}

\textbf{Step 3}. Repeating the proof of \cite[Lemma 2.1]{Wei} in the domain $\mathcal{D}'_t$ and using Step 1 and Lemma \ref{lem-est-away} to treat all boundary terms showing up, we obtain the fixed-time estimates
$$\left\|\nabla \frac{v^\varrho}{\varrho}\right\|_{L^2(\mathcal{D}'_t)} \le \|\Omega\|_{L^2(\mathcal{D}'_t)} + (\G+2)^{C(\H+1)},$$
and
$$\left\|\nabla^2 \frac{v^\varrho}{\varrho}\right\|_{L^2(\mathcal{D}'_t)} \le \|\partial_z\Omega\|_{L^2(\mathcal{D}'_t)} + (\G+2)^{C(\H+1)}.$$

\textbf{Step 4}. Let $\varphi(\varrho)$ be a smooth cutoff function satisfying $\varphi(\varrho) = 1$ for $\varrho<\frac{1}{10}$ and $\varphi(\varrho) = 0$ for $\varrho>\frac{1}{5}$. From the equation \eqref{eq-J}, we obtain the energy estimate
\begin{align} \label{eq-0439}
&\quad \ \frac{d}{dt} \|J \varphi\|_{L^2(\mathcal{D}_t)}^2 + 2\|\varphi\nabla J\|_{L^2(\mathcal{D}_t)}^2 + 4\pi \int_{z_2}^{z_1} J^2(0, z, t) dzdt  \nonumber\\
&=  \int_{\mathcal{D}_t} J^2 \(\Delta - \frac{2}{\varrho} \partial_\varrho\) \varphi^2 + \int_{\mathcal{D}_t} J^2 b \cdot \nabla (\varphi^2) + 2\int_{\mathcal{D}_t} J \varphi^2 (\omega^r \partial_r + \omega^z \partial_z) \frac{v^\varrho}{\varrho}  \nonumber\\
& \quad + \int_{P_2} \partial_z (J^2) \varphi^2 - \int_{P_1} \partial_z (J^2) \varphi^2 + \int_{\mathcal{P}_2} v^z J^2 \varphi^2 -  \int_{\mathcal{P}_1} v^z J^2 \varphi^2 
\end{align}
The first two terms on the right of \eqref{eq-0439} are bounded by $(\G+2)^{C}$ using Lemma \ref{lem-est-away}, and the last four boundary terms are bounded by $(\G+2)^{C(\H+1)}$ using \eqref{eq-P1P2}. Moreover, we have the estimate
\begin{align}
&\quad \ \int_{\mathcal{D}_t} J \varphi^2 (\omega^\varrho \partial_\varrho + \omega^z \partial_z) \frac{v^\varrho}{\varrho}dx  \nonumber \\
&= \int_{\mathcal{D}_t} J \varphi^2 [\nabla \times (v^\theta e_\theta)] \cdot \nabla \frac{v^\varrho}{\varrho}dx  \nonumber \\
&= \int_{\mathcal{D}_t} \varphi^2 v^\theta e_\theta \cdot \( \nabla J \times \nabla \frac{v^\varrho}{\varrho} \)dx  +\int_{\mathcal{D}_t} J v^\theta e_\theta \cdot \( \nabla \varphi^2 \times \nabla \frac{v^\varrho}{\varrho} \)dx + \mathrm{boundary \ terms} \nonumber \\
&\le \|\varphi \nabla J\|_{L^2(\mathcal{D}_t)} \left\|v^\theta \nabla \frac{v^\varrho}{\varrho}\right\|_{L^2(\mathcal{D}_t)} + (\G+2)^{C(\mathcal{H}+1)} .
\end{align}
In the last line we used  Lemma \ref{lem-est-away} and \eqref{eq-P1P2} to control the integral away from the $z$-axis and the boundary terms. Hence, from \eqref{eq-0439} we obtain
\begin{align} \label{J-ee}
&\quad \ \frac{d}{dt} \|J \varphi\|_{L^2(\mathcal{D}_t)}^2 + \|\varphi\nabla J\|_{L^2(\mathcal{D}_t)}^2 + 4\pi \int_{z_1}^{z_2} J^2(0, z, t) dz  \nonumber\\
&\le  \left\|v^\theta \nabla \frac{v^\varrho}{\varrho}\right\|_{L^2(\mathcal{D}_t)}^2 + (\G+2)^{C (\mathcal{H}+1)}.  \end{align}
Let $\psi(\varrho)$ be a smooth cut-off function satisfying $\psi(\varrho) = 0$ for $\varrho > \frac1{10}$ and $\psi(\varrho) = 0$ for $\varrho < \frac1{20}$. Again with the help of Lemma \ref{lem-est-away}, \eqref{J-ee} implies
\begin{align} \label{J-ee-new}
&\quad \ \frac{d}{dt} \|J \varphi\|_{L^2(\mathcal{D}_t)}^2 + \|\varphi\nabla J\|_{L^2(\mathcal{D}_t)}^2 + 4\pi \int_{z_1}^{z_2} J^2(0, z, t) dz  \nonumber\\
&\le  \left\|v^\theta \psi \nabla \frac{ v^\varrho}{\varrho}\right\|_{L^2(\mathcal{D}_t)}^2 + (\G+2)^{C (\mathcal{H}+1)}.
\end{align}
Similarly, using \eqref{eq-Omega} we have
\begin{align} \label{eq-omega-new}
&\quad \ \frac{d}{dt} \|\Omega \varphi\|_{L^2(\mathcal{D}_t)}^2 + 2\|\varphi\nabla \Omega\|_{L^2(\mathcal{D}_t)}^2 + 4\pi \int_{z_1}^{z_2} \Omega^2(0, z, t) dz \nonumber\\
&\le  -4 \int_{\mathcal{D}_t} \psi^2 \frac{v^\theta}{r} J \Omega + (\G+2)^{ C (\mathcal{H}+1)}.
\end{align}
Suppose that 
\begin{equation} \label{eq-sigma-Gamma}
\sigma := [2M(\H+1) \ln (\G+2)]^{-\frac32}  \ge \|\Gamma\|_{L^\infty(\mathcal{D})}
\end{equation}
where $M$ is an absolute constant to be determined later. Next, we need two useful inequalities due to Wei (see \cite[Lemma 2.3]{Wei}): for any $f \in H^1(\varrho d\varrho)$ with support in $\varrho < \frac1{10}$ and for any $z_1<z<z_2$,
\begin{equation} \label{eq-Wei1}
\int \frac{|v^\theta(\varrho,z,t)|}{\varrho} |f|^2 \varrho d\varrho \le \sigma^{-\frac13} \int |\partial_\varrho f|^2 \varrho d\varrho + C  \frac{a(t)^2}{\sigma^{\frac73} K(\sigma)^2} \int |f|^2 \varrho d\varrho,
\end{equation}
and
\begin{equation} \label{eq-Wei2}
\int |v^\theta(\varrho,z,t)|^2 |f|^2 \varrho d\varrho \le \sigma^{\frac23} \int |\partial_\varrho f|^2 \varrho d\varrho + C  \frac{a(t)^2}{\sigma^{\frac43} K(\sigma)^2}  \int |f|^2 \varrho d\varrho,
\end{equation}
where $K(\sigma) = \exp\(\sqrt{\sigma^{-\frac43}-1}-1\)$. We need the assumption on the support of $f$ here, because in our definition of $a(t)$, there is a restriction on the range of $\varrho$. Note that $K(\sigma) \ge \exp\(\frac12 \sigma^{-\frac23}\)$, and for $M$ suitably large,
\begin{equation} \label{eq-K2}
\sigma^\frac73 K(\sigma)^2 \ge \exp \(\frac12 \sigma^{-\frac23}\)  = (\G+2)^{M(\H+1)}.
\end{equation}
Define
\begin{equation}
H(t) = \|J\varphi\|_{L^2(\mathcal{D}_t)}^2 + \frac{\sigma^{\frac23}}{2}\|\Omega \varphi\|_{L^2(\mathcal{D}_t)}^2.
\end{equation}
Applying \eqref{eq-Wei1}, \eqref{eq-Wei2} and Step 3 to the right hand sides of \eqref{J-ee-new} and \eqref{eq-omega-new}, and using \eqref{eq-step2}, \eqref{eq-K2} and Lemma \ref{lem-est-away}, we obtain, for $M$ large enough,
\begin{align} \label{eq-ODE}
&\quad \frac{d}{dt} H(t) \nonumber \\
&\lesssim \frac{a(t)^2}{\sigma^{\frac73} K(\sigma)^2}  \(H(t)+ (\G+2)^{ C (\mathcal{H}+1) }\) +  (\G+2)^{C (\mathcal{H}+1)} \nonumber \\
&\le (\G+2)^{- M (\mathcal{H}+1)} \( \left\|\frac{v^\theta}{\varrho}\right\|_{L^2(\mathcal{D}_t)} \|J\varphi\|_{L^2} + (\G+2)^{ C (\mathcal{H}+1) } \) \(H(t)+ (\G+2)^{ C (\mathcal{H}+1) }\) \nonumber \\
&\quad + (\G+2)^{C (\mathcal{H}+1)}  \nonumber \\
&\le (\G+2)^{-M (\mathcal{H}+1)} \left\|\frac{v^\theta}{\varrho}\right\|_{L^2(\mathcal{D}_t)}  H^{\frac32}(t) + (\G+2)^{(C-M) (\mathcal{H}+1)} \left\|\frac{v^\theta}{\varrho}\right\|_{L^2(\mathcal{D}_t)}  H^{\frac12}(t) \nonumber \\
&\quad  + (\G+2)^{(C-M)(\H+1)}H(t) + (\G+2)^{ C (\mathcal{H}+1)} .
\end{align}
Note that $\left\|\frac{v^\theta}{\varrho}\right\|_{L^2_tL_x^2(\mathcal{D})}^2 \le \int_{Q(\frac{9}{10})} |\nabla v|^2 dxdt \lesssim \G+1$. By Step 1, we have $H(t_0) \lesssim (\G+1)^{C}$. Of course, the above energy estimates and the ODE type estimate \eqref{eq-ODE} is valid only for times $t_0<t<T$. Solving \eqref{eq-ODE}, we obtain that $H(t) \le (\G+2)^{C (\mathcal{H}+1)}$ for all $t_0<t<T$, if $M$ is chosen to be sufficiently large. This implies that $v$ is regular in $\mathcal{D} \cap \{t \le T\}$, and hence $T = 0$. Combining this result with Lemma \ref{lem-est-away}, we deduce that under the assumption \eqref{eq-sigma-Gamma}, $v$ is regular in $\overline{Q(\frac12)}$. 

To finish the proof of Theorem \ref{thm-small-Gamma}, we point out that, in the above arguments, we can use $\int_{Q(\frac{9}{10})}|\nabla v|^2 dxdt$ instead of $\mathcal{H}$, and the former integral is bounded by a multiple of $\G+1$. 
\end{proof}

\subsection{Improved CKN criterion}

\begin{proof}[Proof of Theorem \ref{thm-improved-ckn}]
Under the assumption of \eqref{eq-improved-ckn}, we can apply Lemma \ref{lem-log-grow} to obtain
\begin{equation} \label{eq-AFD-loglog}
(A+F+D)(r) \lesssim (\ln |\ln r|)^{2\mu}
\end{equation}
for sufficiently small $r$.  If $\mu$ is sufficiently small, \eqref{eq-AFD-loglog} enables us to apply the oscillation decay result proved in \cite[Proposition 1.2]{ChenTsaiZhang} or \cite{Seregin2}, which implies that $\Gamma = \varrho v^\theta$ enjoys a modulus of continuity at the origin:
\begin{equation}
|\Gamma(x,t)| \le C(v) \exp(-c |\ln r(x,t)|^\frac12),
\end{equation}
where $r(x,t) = |x| +  \sqrt{-t}$ is small. Here $c$ is an absolute constant independent of $v$. For a sufficiently small $r>0$, we can ensure that
$$\|\Gamma\|_{L^\infty(Q(r))} \le \e ( \G[ v^{(r)}, p^{(r)}] + 2)^{-2},$$
where $ v^{(r)}$, $p^{(r)}$ are defined as in \eqref{eq-scaling}, and $\e$ is any positive constant. Note that $\|\Gamma\|_{L^\infty}$ is a scale-invariant quantity. By Theorem \ref{thm-small-Gamma}, $v$ is regular at the origin $(0,0)$.
\end{proof}

\begin{remark} \label{rem-final}
Let us discuss how Theorem \ref{thm-improved-ckn} implies other slightly supercritical Type I criteria in the axially symmetric case. This is similar to \cite[Section 3]{Seregin3}. Denote
$$\EE(r) = A(r) + E(r) + D(r)$$
and 
$$M^{s,l}(r) = r^{l(1-\frac{3}{s} - \frac{2}{l})} \int_{-r^2}^0 \(\int_{B(r)} |v|^s dx\)^{\frac{l}{s}}dt.$$
Under the constraint
\begin{equation} \label{eq-sl}
\frac12 > \frac{3}{s} + \frac{2}{l} - \frac32 > \max \left\{ \frac{1}{2l}, \frac12 - \frac{1}{l} \right\},
\end{equation}
Seregin and Sverak \cite{SShandbook} showed, via interpolation inequalities and the local energy estimates (similar to those in Section \ref{sec-iter}), that
$$\EE(\theta r) \le \frac12 \EE(r) + C(s,l) \(M^{s,l}(r)\)^{\frac{1}{l\(2-\frac{3}{s}-\frac{2}{l}\)}} + C$$
for an absolute small number $\theta>0$. As a consequence, for any pair $(s, l)$ satisfying \eqref{eq-sl}, the condition
$$\limsup_{r\to 0+}{(\ln |\ln r|)^{-\nu} M^{s,l}(r)} < + \infty,$$ 
implies that
$$\limsup_{r\to 0+}{(\ln |\ln r|)^{-\frac{\nu}{l\(2-\frac{3}{s}-\frac{2}{l}\)}}}E(r) < + \infty.$$
Using Theorem \ref{thm-improved-ckn}, we obtain the one-point regularity criterion: for $\nu$ sufficiently small (depending on $s,l$)
$$\limsup_{r\to 0+}{(\ln |\ln r|)^{-\nu} M^{s,l}(r)} < + \infty \implies (0,0) \notin \S[v].$$
As a further corollary (by taking $s, l$ appropriately), we obtain the improvements of the criteria in \cite{SereginActa, CSYT1, CSTY2}: for   sufficiently small  $\kappa, \sigma>0$ and any $0 \le \alpha \le 1$,
$$|v| \le C_* \frac{ ( \ln |\ln \varrho|)^\kappa \cdot (\ln |\ln (-t)|)^\sigma }{\varrho^\alpha \sqrt{-t}^{1-\alpha}}  \ \mbox{in} \ Q(r) \ \mbox{for some} \ r>0  \implies (0,0) \notin \mathcal{S}[v].$$
In particular, by taking $\alpha = 1$, we have the result of Pan \cite{Pan} (the range of $\kappa$ may be different).
\end{remark}

 \appendix
  \renewcommand{\appendixname}{Appendix~\Alph{section}}

  \section{Appendix}

\begin{proof}[Proof of Lemma \ref{lem-Gamma-bound}]
Recall that for regular solutions, $\Gamma$ satisfies the equation
\begin{equation} \label{eq-Gamma}
\partial_t \Gamma - \Delta \Gamma + \frac{2}{\varrho} \partial_\varrho \Gamma + b \cdot \nabla \Gamma  = 0
\end{equation}
where $b = v^\varrho e_\varrho + v^z e_z$.  It is already proved in \cite{Seregin0} that $\Gamma$ is locally bounded for any axially symmetric local suitable weak solution. Let $\phi$ be a smooth cutoff function compactly supported in $Q(1)$. Let us check that \eqref{eq-Gamma} holds in the weak sense in $Q(1)$ and can be tested with $\Gamma^q \phi^2$ for any $q \ge 0$. It is well-known that, due to the partial regularity Theorem \ref{thm-partial-reg-old}, $v$ is regular away from the $z$-axis. Hence, \eqref{eq-Gamma} holds strongly away from the $z$-axis. Moreover, the projection to the $z$-axis of the singular set $\mathcal{S}[v]$ is of measure zero, that is, there exists a set $\mathcal{S}_z\subset (-1,1)$ of measure zero such that $v$ is regular in $\{z \notin \mathcal{S}_z\} \cap Q(1)$. For $z \notin S_z$, we know that $\Gamma$ is continuous and takes the value $0$  at $r=0$.  Consider a cut-off function $\psi(\varrho)$ such that $\psi = 1$ for $\varrho > 1$ and $\psi = 0$ for $0 \le \varrho<\frac12$, and denote $\psi_\e(\varrho) = \psi\(\frac{\varrho}{\e}\)$. Testing \eqref{eq-Gamma} with $\Gamma^q \psi_\e \phi^2$ we obtain for $-1<T<0$,
\begin{align} \label{eq-0043}
&\quad \ \int_{t=T} \Gamma^{q+1} \psi_\e \phi^2 dx + 2q \int_{t\le T} |\nabla \Gamma^{\frac{q+1}{2}}|^2 \psi_\e \phi^2 dxdt \nonumber \\
&= \int_{t \le T} \Gamma^{q+1} \(\psi_\e\partial_t \phi^2 + \Delta (\psi_\e\phi^2) + \frac{2}{\varrho} \partial_r (\psi_\e\phi^2)\) dxdt +  \int_{t \le T} \Gamma^{q+1} b \cdot \nabla(\psi_\e \phi^2)  dxdt
\end{align}
Note that  as $\e \to 0$, we have
\begin{align}
\int_{\supp \phi} |\Gamma|^{q+1} (|\Delta \psi_\e| + \frac{1}{\varrho}|\nabla \psi_\e|) dxdt \lesssim \frac{1}{\e} \iint \int_{\varrho = \frac{\e}{2}}^{\e} |\Gamma|^{q+1} \mathbf{1}_{\supp \phi} d\varrho dz dt \to 0,
\end{align}
and
\begin{align}
\int_{\supp \phi} |\Gamma|^{q+1} |v^\varrho| |\nabla \psi_\e| dxdt &\lesssim  \iint \int_{\varrho = \frac{\e}{2}}^{\e} |\Gamma|^{q+1} \frac{|v^\varrho|}{\varrho}  \mathbf{1}_{\supp \phi}  \varrho d\varrho dz dt \nonumber \\ 
&\le \|\Gamma\|_{L^\infty( {\supp \phi} )}^{q+1} \iint \int_{\varrho = \frac{\e}{2}}^{\e} \frac{|v^\varrho|}{\varrho}   \mathbf{1}_{\supp \phi} \varrho  d\varrho dz dt \to 0,
\end{align}
In the last line we are using that $\frac{v^\varrho}{\varrho} \in L_t^2L_x^2$ since it is part of the Dirichlet integral of $v$. Hence we can let $\e \to 0$ in \eqref{eq-0043} to obtain
\begin{align} \label{eq-043-new}
&\quad \ \int_{t=T} \Gamma^{q+1} \phi^2 dx + 2q \int_{t\le T} |\nabla \Gamma^{\frac{q+1}{2}}|^2  \phi^2 dxdt \nonumber \\
&= \int_{t \le T} \Gamma^{q+1} \(\partial_t \phi^2 + \Delta \phi^2 + \frac{2}{\varrho} \partial_\varrho \phi^2\) dxdt +  \int_{t \le T} \Gamma^{q+1} b \cdot \nabla \phi^2  dxdt
\end{align}
for $-1<T<0$. 

Next, we use the Moser iteration technique to derive \eqref{eq-Gamma-upper-1}. Let $r_k = \frac34 + 10^{-k}$, $k \ge 1$. Let $\phi_k(|x|,t)$ be a cut-off function supported in $Q(r_k)$, which equals to $1$ in $Q(r_{k+1})$, with the (uniform in $k$) pointwise estimates $|\nabla \phi_k| + |\partial_t \phi_k| \lesssim 10^{k}$ and $|\Delta \phi_k| + |\nabla \phi_k|^2 \phi_k^{-1} \lesssim 10^{2k}$. The last term in \eqref{eq-043-new} can be estimated as
\begin{align} \label{eq-0048}
\int_{t \le T} \Gamma^{q+1} b \cdot \nabla \phi_k^2  dxdt  &\le \|b\|_{L^{\frac{10}{3}}(Q(\frac{9}{10}))} \|\Gamma^{\frac{q+1}{2}} |\nabla \phi_k|^{2} \phi_k^{-1}\|_{L^2}^{\frac12} \| \Gamma^{\frac{q+1}{2}} \phi_k\|_{L^{\frac{10}{3}}}^{\frac32} \nonumber \\
&\le C 10^{4k} (\G+1)^2 \|\Gamma^{\frac{q+1}{2}}\|_{L^2(Q(r_k))}^{2} +  \nonumber \\
&\quad \ + \frac{1}{10} \| \Gamma^{\frac{q+1}{2}} \phi_k\|_{L_t^\infty L_x^2}^2 + \frac{1}{10} \| \nabla(\Gamma^{\frac{q+1}{2}} \phi_k)\|_{L_t^2 L_x^2}^2
\end{align}
Hence, we can deduce from \eqref{eq-043-new} and the known interpolation inequality that
\begin{align}
\|\Gamma^{ \frac{q+1}{2}}\|_{L^{\frac{10}{3}}(Q(r_{k+1}))}^{2} &\lesssim \| \Gamma^{\frac{q+1}{2}} \phi_k\|_{L_t^\infty L_x^2}^2 +  \| \nabla(\Gamma^{\frac{q+1}{2}} \phi_k)\|_{L_t^2 L_x^2}^2 \nonumber \\
&\lesssim 10^{4k} (\G+1)^2 \|\Gamma^{\frac{q+1}{2}}\|_{L^2(Q(r_k))}^{2}
\end{align}
We let $\frac{q+1}{2} = (\frac{5}{3})^{k} $ in the above estimate to get
\begin{equation}
\|\Gamma\|_{L^{2\times (\frac{5}{3})^{k+1}}(Q(r_{k+1}))} \le \( C 10^{2k} (\G+1) \)^{ (\frac35)^{k}} \|\Gamma\|_{L^{2\times (\frac{5}{3})^{k}}(Q(r_{k}))}
\end{equation}
After iteration and letting $k \to \infty$, we arrive at 
\begin{align}
\|\Gamma\|_{L^\infty(Q(\frac34))} &\le \(C(\G+1)\)^{\sum_{k \ge 1} (\frac35)^{k}} 10^{\sum_{k \ge 1}2k (\frac35)^{k}} \|\Gamma\|_{L^{\frac{10}{3}}(Q(r_1))} \nonumber \\
&\lesssim (\G+1)^{\frac32} \|v\|_{L^{\frac{10}{3}}(Q(r_1))} \lesssim  (\G+1)^{2}.
\end{align}

To prove \eqref{eq-Gamma-upper-2}, we notice that in $\varrho \ge \frac1{10}$, $v^{\varrho,z,\theta}$ and $\Gamma$ can be viewed as functions in two variables $\varrho,z$ and the 3D volume element is equivalent to the 2D volume element $\varrho d\varrho dz \sim d\varrho dz$. For instance, using the known interpolation inequality in 2D we have
\begin{align}
\|v\|_{L^4(Q(\frac{9}{10})\cap\{\varrho \ge \frac{1}{10}\})} &\lesssim \|v\|_{L_t^\infty L_x^2(Q(\frac{9}{10})\cap\{\varrho \ge \frac{1}{10}\})} +  \|\nabla v\|_{L_t^2 L_x^2(Q(\frac{9}{10})\cap\{\varrho \ge \frac{1}{10}\})} \nonumber \\
&\stackrel{\eqref{eq-energies-G}}\lesssim (\G+1)^{\frac12}
\end{align}
Hence, instead of \eqref{eq-0048}, we can use the estimate
\begin{align} \label{eq-0048-new}
\int_{t \le T} \Gamma^{q+1} b \cdot \nabla \phi^2  dxdt  &\le \|b\|_{L^{4}(Q(\frac{9}{10})\cap\{r \ge \frac{1}{10}\})} \|\Gamma^{\frac{q+1}{2}} |\nabla \phi| \|_{L^2} \| \Gamma^{\frac{q+1}{2}} \phi\|_{L^{4}} \nonumber \\
&\le C(\G+1) \|\Gamma^{\frac{q+1}{2}} |\nabla \phi| \|_{L^2}^2 +   \nonumber \\
&\quad + \frac{1}{10} \| \Gamma^{\frac{q+1}{2}} \phi\|_{L_t^\infty L_x^2}^2 + \frac{1}{10} \| \nabla(\Gamma^{\frac{q+1}{2}} \phi)\|_{L_t^2 L_x^2}^2
\end{align}
for a cut-off function $\phi$ supported in $Q(\frac{9}{10}) \cap \{\varrho \ge \frac1{10}\}$. Using suitable (axially summetric) cut-off functions and repeating the above iteration argument, it is easy to derive that
\begin{align}
\|\Gamma\|_{L^\infty(Q(\frac34)\cap \{\varrho > \frac14\} )} &\lesssim (\G+1)^{\sum_{k \ge 2} 2^{-k}}  \|\Gamma\|_{L^{4}(Q(\frac{9}{10}) \cap \{\varrho \ge \frac{1}{10}\})} \nonumber \\
&\lesssim (\G+1)^{\frac12} \|v\|_{L^{4}(Q(\frac{9}{10}) \cap \{\varrho \ge \frac{1}{10}\})} \lesssim  \G+1.
\end{align}
\end{proof}

\begin{proof}[Proof of Lemma \ref{lem-est-away}]
Let us denote $W_i = Q(R + (1-\frac{i}{5})(1-R)) \cap \{\varrho > \frac{i}{5} \delta\},\ i = 1,2,3,4,5 $ so that $W_5 \subset W_4 \subset W_3 \subset W_2 \subset W_1$. The proof is divided into three successive steps. The constants in the proof may depend on $R$ and $\delta$.

\textbf{Step 1.} Due to the partial regularity Theorem \ref{thm-partial-reg-old}, $v$ is regular away from the $z$-axis. Recall that, for regular solutions, the function $\Omega = \frac{\omega^\theta}{\varrho} = \frac{\partial_z v^\varrho - \partial_\varrho v^z}{\varrho}$ satisfies the equation 
\begin{equation} \label{eq-Omega}
\partial_t \Omega - \Delta \Omega -\frac{2}{\varrho} \partial_\varrho \Omega + b \cdot \nabla \Omega = \frac{\partial_z \Gamma^2}{\varrho^4}
\end{equation}
where $b = v^\varrho e_\varrho + v^z e_z$. Let $\phi$ be a smooth axially symmetric cut-off function, which equals $1$ in $W_2$ and vanishes in $Q(1) \setminus W_1$. Testing \eqref{eq-Omega} with $\Omega \phi^2$, we obtain for any $-1<T<0$,
\begin{align} \label{eq-043}
&\quad \ \int_{t=T} \Omega^2 \phi^2 dx + 2\int_{t\le T} |\nabla \Omega|^2 \phi^2 dxdt \nonumber \\
&=  \int_{t \le T} \Omega^2 (\partial_t \phi^2 + \Delta \phi^2 - \frac{2}{\varrho} \partial_\varrho \phi^2) dxdt +  \int_{t \le T} \Omega^2 b \cdot \nabla \phi^2  dxdt - \int_{t \le T} \frac{\Gamma^2}{\varrho^4} \partial_z (\Omega \phi^2) \nonumber \\
&=:I_1 + I_2 + I_3
\end{align}
We treat $I_i$ term by term. $I_1$ is the most easy one,
\begin{align} \label{eq-044}
|I_1| \lesssim \int_{\supp \phi} |\Omega|^2 dxdt \lesssim \int_{Q(\frac{9}{10})} |\nabla v|^2 dxdt \stackrel{\eqref{eq-energies-G}}\lesssim \G + 1.
\end{align}
For $I_2$, we note that in $W_1$, $v^{\varrho,z,\theta}$ and $\Omega$ can be viewed as functions in two variables $\varrho,z$ and the 3D volume element is equivalent to the 2D volume element $\varrho d\varrho dz \sim d\varrho dz$. Hence, we can use H\"older's inequality and interpolation inequality in 2D to deduce
\begin{align} \label{eq-045}
|I_2| &\le \|b\|_{L^4(W_1)} \|\Omega \nabla \phi\|_{L^2}  \|\Omega \phi\|_{L^4} \nonumber \\
&\lesssim (\G+1) \(\|\Omega \phi\|_{L^\infty_t L^2_x} + \|\nabla(\Omega \phi)\|_{L_t^2L_x^2} \)  \nonumber \\
&\le C\delta^{-1} (\G+1)^2 +  \delta\(\|\Omega \phi\|_{L^\infty_t L^2_x}^2 + \|\nabla(\Omega \phi)\|_{L_t^2L_x^2}^2 \)
\end{align}
For $I_3$, we have
\begin{align} \label{eq-046}
|I_3| &\le \int_{t \le T} \frac{\Gamma^2}{\varrho^4} |\partial_z (\Omega\phi)| \phi + \int_{t \le T} \frac{\Gamma^2}{\varrho^4} |\partial_z \phi| |\Omega| \phi \nonumber\\
 &\le C\delta^{-1} \|v^\theta\|_{L^4(W_1)}^4 + \delta\|\nabla (\Omega\phi)\|^2_{L_t^2L_x^2} + C\|v^\theta\|_{L^4(W_1)}^2 \|\Omega\|_{L^2_tL_x^2(W_1)} \nonumber \\
&\le C\delta^{-1} (\G+1)^2 + \delta \|\nabla (\Omega\phi)\|^2_{L_t^2L_x^2(Q(1))} + C(\G+1)^{\frac32} \nonumber \\
&\le C\delta^{-1} (\G+1)^2 + \delta \|\nabla (\Omega\phi)\|^2_{L_t^2L_x^2(Q(1))}
\end{align}
Combining \eqref{eq-043}--\eqref{eq-046}, we obtain
\begin{align}
\sup_{-1<t<0} \int \Omega^2 \phi^2 dx + \int_{Q(1)} |\nabla \Omega|^2 \phi^2 dxdt \lesssim (\G+1)^2
\end{align} 
In particular, we have
\begin{align} \label{eq-0424}
\sup_{-1<t<0} \int (\omega^\theta)^2 \phi^2 dx \lesssim (\G+1)^2
\end{align}

\textbf{Step 2.} Using the  div-curl system
\begin{align}
&\partial_z v^\varrho -\partial_\varrho v^z =\omega^\theta \nonumber\\
&\partial_\varrho v^\varrho  + \partial_z v^z = - \frac{v^\varrho}{\varrho}
\end{align}
in $W_2$ and (two-dimensional) Sobolev embedding, we obatin, for any $q < +\infty$,
\begin{align}
\|b \|_{L^\infty_t L^q_x(W_3)} \lesssim_q \|\omega^\theta \phi\|_{L^\infty_t L^2_x(W_2)} + \|b\|_{L^\infty_t L^2_x(W_2)} \stackrel{\eqref{eq-0424}}\lesssim \G+1.
\end{align}
By Lemma \ref{lem-Gamma-bound}, we have $\|v^\theta\|_{L^\infty(W_3)} \lesssim_q \G+1$. Then, using the standard decomposition of pressure (see the proof of Lemma \ref{lem-multiplicative}), we obatin
\begin{align}
\|p\|_{L^{\frac32}_t L^{\frac{q}{2}}_x(W_4)} &\lesssim \|v \|_{L^\infty_t L^q_x(W_3)}^2 + \|p\|_{L^{\frac32}_t L^{\frac32}_x(W_3)} \lesssim (\G+1)^2.
\end{align}

\textbf{Step 3.} For $0 < r < \frac{1}{10}\min \left\{ \delta, (1-R) \right\}$ and $w \in W_5$, we have $Q(r,w) \subset W_4$, and 
\begin{align}
&\quad  \ \frac{1}{r^2}\int_{Q(r,w)} \(|v|^3 + |p|^\frac32\) dxdt  \nonumber \\
&\lesssim \frac{1}{r^2} \(\|v\|_{L^\infty_t L^q_x(W_4)}^3 r^{5-\frac{9}{q}} + \|p\|_{L^{\frac32}_t L^{\frac{q}{2}}_x}^\frac32 r^{3-\frac{9}{q}}\) \nonumber \\
&\lesssim_q (\G+1)^3 \( r^{3-\frac{9}{q}} +  r^{1-\frac{9}{q}}\) \lesssim (\G+1)^3 r^{1-\frac{9}{q}}.
\end{align}
Take $q = 36$ and $r < c (\G+1)^{-4}$ for a small constant $c$, then 
\begin{align}
\frac{1}{r^2}\int_{Q(r,w)} \(|v|^3 + |p|^\frac32\) dxdt  \le \e_*.
\end{align}
Using Theorem \ref{thm-ereg} and rescaling, we deduce that $v$ is regular in $Q(\frac{r}{2},w)$ with the estimates
\begin{equation}
\|\nabla_x^m v\|_{L^\infty(Q(\frac{r}{2},w))} \lesssim_m r^{-m-1} \lesssim (\G+1)^{4m+4}, \quad m \ge 0.
\end{equation}
Since $w \in W_5$ is arbitrary, the theorem is proved.

\end{proof}

\section*{Acknowledgement}
We would like to thank professor Gang Tian and professor Dongyi Wei for helpful discussions. The authors were in part supported by NSFC (Grant No. 11725102), Sino-German Center
Mobility Programme (Project No. M-0548) and Shanghai Science and Technology Program (Project
No. 21JC1400600 and No. 19JC1420101).

\bibliographystyle{plain}
\bibliography{ns}

\end{document}